\documentclass[12pt, twoside]{article}   	
\usepackage[utf8]{inputenc}
\usepackage{geometry}               		
\usepackage{graphicx}				

\usepackage{xcolor}								
\usepackage{amssymb}

\newcommand{\norm}[1]{\left\lVert#1\right\rVert}
\usepackage{amsmath}
\usepackage{breqn}
\usepackage[T1]{fontenc}
\usepackage{mathtools}
\usepackage{color}
\usepackage{mathrsfs}
\usepackage{amsthm}
\usepackage{amsmath, amssymb, amsthm, enumerate,graphicx,textcomp,fullpage, cite}
\usepackage{listings}
\usepackage{color} 
\definecolor{mygreen}{RGB}{28,172,0} 
\definecolor{mylilas}{RGB}{170,55,241}

\usepackage{tikz}
\usepackage{esvect}
\usepackage{hyperref}
\definecolor{red}{rgb}{1,0,0}

\newtheorem{theorem}{Theorem}[section]
\newtheorem{corollary}{Corollary}
\newtheorem{lemma}{Lemma}
\newtheorem{remark}{Remark}

\newtheorem{definition}{Definition}[section]
\newtheorem{proposition}{Proposition} 

\title{Global Well-Posedness and Blow-Up for the fifth order $L^2-$critical KP-I equation}
\author{Francisc Bozgan \thanks{NYUAD Research Institute, New York University Abu Dhabi, Saadiyat Island, PO Box 129188, Abu Dhabi, United Arab Emirates, {\sf ftb208@nyu.edu}}}
\date{}

\begin{document}

\maketitle
\begin{abstract}
    In the current paper, we investigate the fifth order modified KP-I eqaution, namely 
    \begin{equation*}
     \partial_t u-\partial_{x}^{5}u-\partial_{x}^{-1}\partial_{y}u+\partial_{x}(u^3)=0.
    \end{equation*}
    This equation is $L^2$ critical and we prove on $\mathbb{R}\times\mathbb{R}$ that it is globally well posed in the natural energy space if the $L^2$ norm of the initial data is less the $L^2$ norm of the ground state associated to this equation. We also find a subspace of the natural energy space associated to this equation where we have local well-posedness, nevertheless if the initial data is sufficiently localized we obtain blow-up. On $\mathbb{R}\times \mathbb{T},$ we prove global well-posedness in the energy space for small data. 
\end{abstract}
\section{Introduction}
We start with the equation 
\begin{equation}\label{fifthKPI}
\begin{cases}
&u_t-\partial_{x}^{5}u-\partial_{x}^{-1}\partial_{y}^{2}u+\partial_{x}(u^{p+1})=0 \\
&u(0,x,y)=u_0(x,y)\in H^s(\mathbb{R}^2)
\end{cases}
\end{equation}
which is the fifth order generalized Kadomtsev-Petviashvili-I equation for $p\geq1$ and $p=\frac{m}{n}$ where $(m,n)=1,$ $n$ is an odd integer. For $p=2$, it appears as a model of sound waves in antiferromagnetic materials, see \cite{FalkovitchTuritsyn}. This is a modified version of an equation that is part of the hierarchy of KP equations, namely 
\begin{equation}\label{hierarchy}
    \partial_{t}+(-1)^{\frac{l+1}{2}}\partial_{x}^{l}u\pm \partial_{x}^{-1}\partial_{y}u+u\partial_{x}u=0,
\end{equation}
with the classical KP equations corresponding to $l=3$ and depending on the sign of the weak transverse effect term, we call the equation \eqref{hierarchy} KP-I equation if the sign is $-$ and KP-II equation if the sign is $+.$  Here the operator $\partial_x^{-1}$ is defined via Fourier transform, $\widehat{\partial_x^{-1}f}(\xi,\eta)=\frac{1}{i\xi}\widehat{f}(\xi,\eta).$ These equations are well-studied, especially the classical KP equations. It is a model for the propagation of weakly nonlinear dispersive long waves on the surface of a fluid, when the wave motion is essentially one-directional with weak transverse effects along the y-axis, see \cite{KadomtsevPetviashvili}, \cite{AbramyanStepanyats}. We refer to \cite{KleinSaut} to an extensive overview on these equations. More specifically, in \cite{LiXiao} it was proved that the KP-I equation \eqref{hierarchy} for $l=5$ is globally well-posed in $L^2(\mathbb{R}^2)$ (see also \cite{SautTzvetkov}).  

In terms of local well-posedness of \eqref{fifthKPI}, we state the best result about this which appears in \cite{Esfahani}: 
\begin{proposition}
    Let $u_0 \in X_s , s\geq 5$. Then there exists $T>0$ such that \eqref{fifthKPI} has a unique solution $u(t)$ with $u(0) = u_0$ satisfying 
$$u\in C([0,T);X_s )\cap C^{1}([0,T);H^{s-5}(\mathbb{R}^2)).$$ 
Moreover, we have the conservation laws: 
$$\|u(t)\|_{L^2}=\|u_0\|_{L^2} \mbox{  and  }E(u(t))=\frac{1}{2}\int_{\mathbb{R}^2}[(\partial_{x}^{2}u)^2+(\partial_{x}^{-1}\partial_{y}u)^2](t)dxdy-\frac{1}{p+2}\int_{\mathbb{R}^2}u^{p+2}(t)=E(u_0).$$
\end{proposition}

We have both scaling and translation invariance for this equation, in particular we have that $u_{\lambda}(t,x,y)=\lambda^{\frac{4}{p}}u(\lambda^5t,\lambda x, \lambda^3 y)$ solves equation \eqref{fifthKPI}, hence 
$$\|u_{\lambda}\|_{\dot{H}^{s_1,s_2}}=\lambda^{\frac{4}{p}+s_1+3s_2-2}\|u\|_{\dot{H}^{s_1,s_2}},$$
and we notice that for $\|u_{\lambda}\|_{L^2}=\lambda^{\frac{4}{p}-2}\|u_0\|_{L^2}$ and $E(u_{\lambda})=\lambda^{\frac{8}{p}}E(u)=\lambda^{\frac{8}{p}}E_0.$ Thus, for $p=2$ we observe that it is the $L^2$-critical case for the equation \eqref{fifthKPI}.

In this paper, we improve this local well-posedness result for $p=2.$ We have the following theorem: 
\begin{theorem}\label{mainthm}
    We define $E^s=\{u\in \mathcal{S}':\|(1+D_x)^su\|_{L^2}+\|\partial_{x}^{-1}\partial_{y}u\|_{L^2}<\infty\}. $. Then the equation \eqref{fifthKPI} with initial data $u_0\in E^s$ is locally well-posed for $s>\frac{3}{2}.$ 
\end{theorem}
In order to prove this, we employ the energy method that was used to prove local well-posedness result for the modified third order KP-I that appeared in \cite{KenigZiesler} and \cite{KenigZiesler2}. 

By a solitary-wave solution of \eqref{fifthKPI}, we mean a traveling-wave solution of \eqref{fifthKPI} in the form $u(t, x, y) = \varphi(x-ct, y)$ with $u \rightarrow 0$ as
$x^2 + y^2 \rightarrow \infty$. And $\varphi_c$ is a solution of the equation
\begin{equation}\label{groundstate}
    \begin{cases}
    -c\varphi_c-\partial_{x}^{4}\varphi_c-\partial_{x}^{-2}\partial_{y}^{2}\varphi_c+\varphi_{c}^{p+1}=0\\
    \varphi_c \in E^2, \varphi_c\neq 0.
\end{cases}
\end{equation}

The existence of solitary waves for equation \eqref{fifthKPI} was proved in \cite{deBouardSaut1} and \cite{EsfahaniLevandovsky}. 
We define $L_c(u)=E(u)+\frac{c}{2}\|u\|_{L^2}.$ We define the ground state as follows. 

\begin{definition}\label{defgroundstate}
  Let $\Gamma_c$ be the set of the solutions of \eqref{groundstate}; namely,
$\Gamma_c =\{\phi \in E^2 |L'_c(\phi)=0,\phi\neq 0\}$, and let $G_c$ be the set of the ground states of \eqref{groundstate}; that is
$$G_c =\{\varphi \in \Gamma_c |L_c(\varphi)\leq L_c(\phi),\forall \phi\in \Gamma_c\}.$$ 
We also define $d_c=\inf_{\varphi \in G_c}L_{c}(\varphi).$
\end{definition}
From (\cite{EsfahaniLevandovsky}, Theorem 2.2) we know that $G_c$ is not empty. We know many properties of the ground state of \eqref{groundstate}, for example, $\varphi \in G_1$ is orbitally unstable for $p>2$ from \cite{EsfahaniLevandovsky}. Nevertheless, we do not know the uniqueness of the ground states of equation \eqref{fifthKPI}.  Next, we state a lemma in order to show that $d_c$ does not depend specifically on the ground state. 
\begin{lemma}\label{pohozaev}
    Let $\varphi\in G_1.$ Then we have the following equalities: 
    $$\int_{\mathbb{R}^2}(\partial_{x}^{2}\varphi)^2dxdy=\int_{\mathbb{R}^2}(\partial_{x}^{-1}\partial_{y}\varphi)^2dxdy=\frac{p}{4}\int_{\mathbb{R}^2}\varphi^2=\frac{p}{2(p+2)}\int_{\mathbb{R}^2}\varphi^{p+2}.$$
\end{lemma}
 This lemma is a consequence of Pohozaev identities. Hence, if $\varphi_1 \in G_1 $, we have that $L_1(\varphi_1)=\frac{p}{4}\|\varphi_1\|_{L^2}^{2}=d_1$ is independent of the choice $\varphi_1.$ We can state the corollary of the above theorem: 
\begin{corollary}\label{globalWP}
    If $u_0\in E^2$ such that $\|u_0\|_{L^2}<\|\varphi_1\|_{L^2},$ where $\varphi_1\in G_1,$ then we have that the solution of \eqref{fifthKPI} for $p=2$ with initial data $u_0$ is globally well-posed in $E^2.$ 
\end{corollary}

We observe that this resembles the global well-posedness result in $H^1(\mathbb{R})$ for the $L^2$ critical case of the generalized KdV equation, see \cite{Weinstein}, and the global well-posedness result in $H^1(\mathbb{R}^2)$ for the $L^2$ critical case of the generalized Zakharov-Kuznetsov equation, see \cite{LinaresPastor}. 

We also remind the blow-up result that we know about equation \eqref{fifthKPI}. Using the methods developed to prove blow-up and strong instability of ground states for the third order KP-I equation by \cite{Liu1}, \cite{Liu2}, it is proven in \cite{EsfahaniLevandovsky} the following blow-up result: 
\begin{proposition}
    Let $Q(u)=\|\partial_{x}^{-1}\partial_{y}u\|_{L^2(\mathbb{R}^2)}^{2}-\frac{p}{2(p+2)}\int_{\mathbb{R}^2}u^{p+2}$ and denote $Q^{-}=\{u \in X_s: L_1(u)<0, Q(u)<0\}.$ If $u_0 \in X_s, |y|u_0 \in L^2(\mathbb{R}^2), u_0\in Q^{-},$ then the solution $u(t)$ of the equation \eqref{fifthKPI} with $p\geq 4$ with initial data $u_0$ blows-up in finite time, namely there exists $T_0<\infty$ such that 
    $$\lim_{t\rightarrow T_{0}^{-}}\|\partial_{y}u(t)\|_{L^2(\mathbb{R}^2)}=\infty.$$
\end{proposition}

    We will improve this result for $p\geq 2,$ using some different invariant spaces than in \cite{EsfahaniLevandovsky}. By employing this generalization, if we restrict the initial data to a smaller space than $E^2$ but larger than $X_s$, we will obtain blow-up of a solution for the equation \eqref{fifthKPI} for $p=2$ for any initial data, despite being locally well-posed in the said subspace. We bring together all these remarks in the following theorem. 
\begin{theorem}
    There exists a subspace $X_s \subset \tilde{E}^2 \subset E^2$ such that if $u_0 \in \tilde{E}^2,$ then there exists $T=T(\|u_0\|_{L^2}, \|u_0\|_{\tilde{E}^2})$ such that a solution $u(t)$ of \eqref{fifthKPI} for $p=2$, exists in $C([0,T), \tilde{E}^2)$ and 
    $$\lim_{t\rightarrow T^{-}}\|u_0\|_{\tilde{E}^2}=+\infty.$$
\end{theorem}

For the case $\mathbb{R}\times\mathbb{T},$ we mention the result from \cite{KenigIonescu} which proves global well-posedness on $\mathbb{R}\times\mathbb{T}$ of \eqref{fifthKPI} with $p=1$ in the energy space $Z^2(\mathbb{R}\times\mathbb{T})=\{g:\mathbb{R}\times\mathbb{T}\rightarrow\mathbb{T}:\|g\|_{Z^2}=\|(1+\xi^2+\frac{n^2}{\xi^2})\hat{g}(\xi,n)\|_{L^2(\mathbb{R}\times\mathbb{T})}<\infty\}$. We improve their localized in time Strichartz estimates to prove global well-posedness for the equation \eqref{fifthKPI} with $p=2$ on $\mathbb{R}\times\mathbb{T}$ in $Z^2(\mathbb{R}\times\mathbb{T).}$ In order to prove that, we prove an intermediary local well-posedness result for small data. We define 
$$Z^s(\mathbb{R}\times\mathbb{T})=\{g:\mathbb{R}\times\mathbb{T}\rightarrow\mathbb{T}:\|g\|_{Z^s}=\|(1+|\xi|^s+\frac{n^2}{\xi^2})\hat{g}(\xi,n)\|_{L^2(\mathbb{R}\times\mathbb{T})}<\infty\}.$$
\begin{theorem} 
 The initial value problem (\ref{eqmkp5RT}) is locally well-posed in $Z^{s}(\mathbb{R}\times\mathbb{T}),s>\frac{31}{20}.$ More precisely, given $u_0 \in Z^{s}(\mathbb{R}\times\mathbb{T}),s>\frac{31}{20}$, with $\|u_0\|_{Z^s}$ small enough, there exists $T=T(\|u_0\|_{Z^{s}})$ and a unique solution $u$ to the IVP such that $u \in C([0,T]:Z^{s}(\mathbb{R}\times \mathbb{T})).$ Moreover, the mapping $u_0 \rightarrow u \in C([0,T]:Z^{s}(\mathbb{R}\times \mathbb{T}))$ is continuous. Also, we have conservation of the quantities 
 $$M(u(t))=\int_{\mathbb{R}\times\mathbb{T}}u(t)^2dxdy \text{  and  }E_{(s)}(u(t))=\frac{1}{2}\int_{\mathbb{R}\times\mathbb{T}}|D_{x}^su(t)|^2+[\partial_{x}^{-1}\partial_{y}u(t)]^2-\frac{1}{2}\int_{\mathbb{R}\times\mathbb{T}}u(t)^4.$$
 \end{theorem} 
As a corollary, we obtain global well-posedness in $Z^2$ with small data. 

We organize the paper as the following: in Section \ref{RR}, we prove global well-posedness in $\mathbb{R}\times \mathbb{R},$ we follow with the transverse blow-up in Section \ref{BU} and we conclude with the global well-posedness result in $\mathbb{R}\times \mathbb{T}$ in Section \ref{RT}.

\section{Well-posedness for the $L^2$ critical fifth order KP-I equation }\label{RR}
\subsection{Local Well-Posedness Result}
In this section, we will deal with the following equation 
\begin{equation}\label{fifthKPI2}
\begin{cases}
&u_t-\partial_{x}^{5}-\partial_{x}^{-1}\partial_{y}^{2}u+\partial_{x}(u^{3})=0 \\
&u(0,x,y)=u_0(x,y)\in H^s(\mathbb{R}^2)
\end{cases}
\end{equation}
We begin by introducing the space, for $s \in \mathbb{R}$, 
$$H_{-1}^{s}(\mathbb{R}^2)=\{ u \in \mathcal{S}'(\mathbb{R}^2): \|u\|_{H_{-1}^{s}(\mathbb{R}^2)}<\infty\}$$
where 
$$\|u\|_{H_{-1}^{s}(\mathbb{R}^2)}=\Big(\int \int (1+|\xi|^{-1})^2(1+|\xi|^2+|\mu|^2)^s|\hat{u}(\xi,\mu)|^2d\xi d\mu\Big)^{\frac{1}{2}}.$$

The proof of Theorem \eqref{mainthm} follows from the following proposition. 

\begin{proposition}\label{mainproposition}
    Suppose $u_0\in H_{-1}^{s}(\mathbb{R}^2)$, for $s_0$ sufficiently large, is a solution to \eqref{fifthKPI}, for $t\in [0,T_0].$ Suppose that there exists $s, \tilde{T}=\tilde{T}(\|u_0\|_{E^s})$ and $C=C_s$ such that, whenever $T_0\leq \tilde{T},$ 
    \begin{itemize}
        \item[a)] $\sup_{0<t<T_0}\|u(t)\|_{E^s}\leq C\|u_0\|_{E^s}$
          
           and 
        \item[b)]  $\|\partial_xu\|_{L^{2}_{T_0}L^{\infty}_{xy}}+\|u\|_{L^{2}_{T_0}L^{\infty}_{xy}}\leq C_{\varepsilon}T_{0}^{\varepsilon}(1+\|u_0\|_{E^s}+\|u_0\|_{E^s}^{3})\|u_0\|_{E^s}.$ 
         \end{itemize}
    Then equation \eqref{fifthKPI2} is locally well posed in $E^s.$     

\end{proposition}

The proof of this is standard, using a parabolic regularization and arguing as in \cite{IorioNunes} and the continuity of the flow map follows from the well-known Bona-Smith argument as in \cite{BonaSmith}. It remains to prove the a priori estimates a) and b) in the proposition. We follow the same ideas used in \cite{Kenig}. First, we have an energy estimate. 

\begin{lemma}\label{energyestimates}
    Let $u$ be a solution to \eqref{fifthKPI2}, with $u_0 \in H_{-1}^{s_0}(\mathbb{R}^2),$ for $s_0$ sufficiently large. Then, for $s\geq 1,$ we have 
    $$\partial_t(\|\langle D_x\rangle^su\|_{L^{2}_{xy}}^2+\| D_{y}D^{-1}_{x}u\|_{L^{2}_{xy}}^2)\leq C(\|u\|_{L^{\infty}_{xy}}^{2}+\|\partial_x u\|_{L^{\infty}_{xy}}^{2})\|\langle D_x\rangle^su\|_{L^{2}_{xy}}^2+\| D_{y}D^{-1}_{x}u\|_{L^{2}_{xy}}^2).$$
\end{lemma}

We define the semigroup operator associated to the equation \eqref{fifthKPI2} by $$W(t)=\int \int e^{i(x\xi+y\mu)}e^{it(\xi^5+\frac{\mu^2}{\xi})}d\xi d\mu.$$

Fix a function $\eta:\mathbb{R}\rightarrow [0,1]$ such that it is smooth, supported in $\{|\xi| \leq \frac{8}{5}\}$ and equals to $1$ for $|\xi| \leq \frac{5}{4}$. 
Let $\Phi(\xi) = \eta(|\xi|/2)-\eta(|\xi|)$. For $N \in 2^{\mathbb{Z}}$ let $\Phi_N(\xi)=\Phi(\frac{\xi}{N})$ and define $P_N$ by $\widehat{P_N u}(\xi,\mu) = \Phi_N (\xi)\hat{u}(\xi, \mu).$ 

\begin{lemma} \label{StrichartzEstimates}
    We have
   $$ \|W(t)\phi\|_{L^{q}_{T}L^{r}_{xy}}\leq C\|\phi\|_{L^{2}_{xy}}$$
   for $\frac{5}{4q}+\frac{1}{r}=\frac{1}{2},$ with $ 2\leq r<\infty,$ and with $C$ independent of $T.$      
\end{lemma}
\begin{proof} 
First we prove for any $N\in 2^{\mathbb{Z}}$, we have
\begin{equation}\label{strich1}
   \|W(t)P_N\phi\|_{L^{q}_{T}L^{r}_{xy}}\lesssim N^{\beta(q,r)}\|\phi\|_{L^{2}_{xy}} 
\end{equation}

By $TT^*$-method, \eqref{strich1} is equivalent to 
\begin{equation}\label{strich2}
    \norm{\int W(t-s)P_Nfds}_{L^{q}_{T}L^{r}_{xy}}\lesssim N^{2\beta(q,r)}\|f\|_{L^{q'}_{T}L^{r'}_{xy}}.
\end{equation}
We define $W(t)P_N\phi=G(\cdot,\cdot,t)\star \phi,$ where $\star$ denotes the convolution with respect to the spatial variables and $G$ is defined by 
$$G(x,y,t)=\int_{\mathbb{R}^2}e^{i(x\xi+y\mu)}e^{it(\xi^5+\frac{\mu^2}{\xi}})\psi(\frac{\xi}{N})d\xi d\mu.$$
Integrating $\mu,$ we get $$G(x,y,t)=\int_{\mathbb{R}}e^{it\xi^5}|t|^{-\frac{1}{2}}|\xi|^{\frac{1}{2}}\psi_N(\xi) e^{it(x+\frac{y^2}{t})},$$
hence 
$$\|G\|_{L^{\infty}_{xy}}=\sup_{x,y\in \mathbb{R}}|G(x,y,t)|=|t|^{-\frac{1}{2}}\sup_{x\in \mathbb{R}}\Big|\int_{\mathbb{R}}|\xi|^{\frac{1}{2}}e^{it\xi^5}e^{it\xi}\psi\Big(\frac{\xi}{N}\Big)\Big|$$
By \cite{GuoPengWang}, for $\theta \in [0,\frac{1}{2}]$ we have that 
$$\sup_{x\in \mathbb{R}}\Big|e^{it\xi^5}e^{ix\xi}\psi\Big(\frac{\xi}{N}\Big)d\xi\Big|\lesssim |t|^{-\theta}N^{1-5\theta}.$$
Thus, using $|\xi|\sim N$, we obtain $\|G(\cdot,\cdot,t)\|_{L^{\infty}_{xy}}\lesssim |t|^{-\frac{1}{2}-\theta}N^{\frac{3}{2}-5\theta}$ and therefore by Young's inequality, 
$$\|G\star \phi\|_{L^{\infty}_{xy}}\lesssim |t|^{-\frac{1}{2}-\theta}N^{\frac{3}{2}-5\theta}\|\phi\|_{L^{1}_{xy}}.$$
Also, using the semigroup properties of $W(t)$, we get 
$$\|G\star \phi\|_{L^{2}_{xy}}=\|W(t)P_N\phi\|_{L^{2}_{xy}}\lesssim \|\phi\|_{L^{2}_{xy}}.$$
Using interpolation, $$\|W(t)P_N\phi\|_{L^{r}_{xy}}\lesssim |t|^{-(\frac{1}{2}+\theta)(1-\frac{2}{r})}N^{(\frac{3}{2}-5\theta)(1-\frac{2}{r})}\|\phi\|_{L^{r'}_{xy}}.$$ 
By Hardy-Sobolev inequality, we obtain 
$$\norm{\int_{\mathbb{R}^2}W(t-s)P_Nf}_{L^{q}_{T}L^{r}_{xy}}\lesssim \norm{\int\|W(t-s)P_Nf\|_{L^{r}_{xy}}ds}_{L^{q}_{T}}$$
$$\lesssim \norm{\int |t-s|^{(-\frac{1}{2}-\theta)(1-\frac{2}{r})}N^{(\frac{3}{2}-5\theta)(1-\frac{2}{r})}\|f\|_{L^{r'}_{xy}}ds}_{L^{q}_{T}}\lesssim N^{(\frac{3}{2}-5\theta)}\|f\|_{L^{r'}_{xy}},$$
for $\theta=\frac{\frac{2}{q}+\frac{1}{r}-\frac{1}{2}}{1-\frac{2}{r}},$ which implies $(\frac{3}{2}-5\theta)(1-\frac{2}{r})=8(\frac{1}{2}-\frac{1}{r}-\frac{1}{q})-\frac{2}{q}.$ Since $\theta\in [0,\frac{1}{2}],$ we need $$\frac{1}{2}-\frac{1}{q}-\frac{1}{r}\in \Big[0,\frac{1}{q}\Big].$$
We obtain 
$$\norm{\int_{\mathbb{R}^2}W(t-s)P_Nf\|_{L^{q}_{T}L^{r}_{xy}}}\lesssim N^{8(\frac{1}{2}-\frac{1}{q}-\frac{1}{r})-\frac{2}{q}}\|f\|_{L^{q'}_{T}L^{r'}_{xy}},$$
hence 
$$\|W(t)P_Nf\|_{L^{q}_{T}L^{r}_{xy}}\lesssim N^{4(\frac{1}{2}-\frac{5}{4q}-\frac{1}{r})}\|f\|_{L^{2}_{xy}}.$$ 
Since, $\frac{1}{2}=\frac{5}{4q}+\frac{1}{r}$, we conclude that 
$$\|W(t)P_Nf\|_{L^{q}_{T}L^{r}_{xy}}\lesssim\|f\|_{L^{2}_{xy}}.$$
\end{proof}

We adapt now Lemma $5.2$ from \cite{KenigZiesler}. 

\begin{lemma} 
    For all $\varepsilon>0,T>0$ there is a constant $C(\varepsilon)$ such that 
    $$\|W(t)\phi\|_{L^{2}_{T}L^{\infty}_{xy}}\leq C_{\varepsilon}T^{\varepsilon+\frac{1}{10}}(\|\phi\|_{L^{2}_{xy}}+\|D_{x}^{3\varepsilon}\phi\|_{L^{2}_{xy}}+\|D_{y}^{3\varepsilon}\phi\|_{L^{2}_{xy}}).$$
\end{lemma}
\begin{proof}
    We begin by recalling Lemma 3 of \cite{KenigZiesler2}, which says that 
    \begin{equation}\label{multiplierestimate}
      \|f\|_{L^{\infty}({\mathbb{R}^2})}\leq C\{\|f\|_{L^{\delta}({\mathbb{R}^2})}+\|D_{x}^{\delta}f\|_{L^{\delta}({\mathbb{R}^2})}+\|D_{y}^{\delta}f\|_{L^{\delta}({\mathbb{R}^2})}\},  
    \end{equation}
     for $p_{\delta}>\frac{2}{\delta}.$
    By H\"{o}lder's inequality in $t$, \eqref{multiplierestimate} and Lemma \ref{StrichartzEstimates} we obtain 
    $$\|W(t)\phi\|_{L^{2}_{T}L^{\infty}_{xy}}\leq C_{\varepsilon}T^{\varepsilon+\frac{1}{10}}\|W(t)\phi\|_{L^{q_{\varepsilon}}_{T}L^{\infty}_{xy}}$$
    $$\leq C_{\varepsilon}T^{\varepsilon+\frac{1}{10}}\{\|W(t)\phi\|_{L^{q_{\varepsilon}}_{T}L^{p_{\varepsilon}}_{xy}}+\|D_{x}^{3\varepsilon}W(t)\phi\|_{L^{q_{\varepsilon}}_{T}L^{p_{\varepsilon}}_{xy}}+\|D_{y}^{3\varepsilon}W(t)\phi\|_{L^{q_{\varepsilon}}_{T}L^{p_{\varepsilon}}_{xy}}\}$$
    $$\leq C_{\varepsilon}T^{\varepsilon+\frac{1}{10}}(\|\phi\|_{L^{2}_{xy}}+\|D_{x}^{3\varepsilon}\phi\|_{L^{2}_{xy}}+\|D_{x}^{3\varepsilon}\phi\|_{L^{2}_{xy}})$$
    for $\frac{1}{q_{\varepsilon}}=\frac{1}{2}-\varepsilon-\frac{1}{10}$ and $p_{\varepsilon}=\frac{4}{5\varepsilon}>\frac{2}{3\varepsilon}.$
\end{proof}

We now state the a priori estimates on $w.$ 
\begin{lemma} \label{aprioriestimate}
    Let $T\in (0,1), \varepsilon>0.$ Assume that $w \in C([0,T]:H^{3}_{-1}(\mathbb{R}^2))$ is a solution of the linear equation 
    $$\partial_t w-\partial_{x}^{5}w-\partial_{x}^{-1}\partial_{y}^{2}w=F.$$
    Then for $\theta \in (0,\frac{5}{4}),$ 
    $$\|\partial_x w\|_{L^{2}_{T}L^{\infty}_{xy}}\leq C_{\varepsilon}T^{\varepsilon+\frac{1}{10}}\Big[\sup_{0<t<T}\{\|\langle D_x\rangle^{\frac{3}{2}+4\varepsilon} w\|_{L^{2}_{xy}}+\|\langle D_x\rangle^{\frac{3}{2}+\varepsilon} D_{y}^{3\varepsilon} w\|_{L^{2}_{xy}}\}$$
    $$+T^{1-\frac{4\theta}{5}}\{\|\langle D_x\rangle^{\frac{1}{2}+4\varepsilon}F\|_{L^{\frac{5}{4\theta}}_{T}L^{2}_{xy}}+\|\langle D_x\rangle^{\frac{1}{2}+\varepsilon}D_{y}^{3\varepsilon}F\|_{L^{\frac{5}{4\theta}}_{T}L^{2}_{xy}}\}\Big].$$
\end{lemma}

The proof is similar to Lemma 5.3 in \cite{KenigZiesler}, hence we omit the details.

\begin{corollary} \label{boundforL2T}
    Let $T,w$ be as in Lemma \ref{aprioriestimate}. Let $F=-w^2\partial_xw.$ Then, for all $\varepsilon>0, \theta \in (0,\frac{5}{4}), s>\frac{\frac{3}{2}+4\varepsilon}{1-3\varepsilon},$ we have
    $$\|\partial_xw\|_{L^{2}_{T}L^{\infty}_{xy}}\lesssim C_{\varepsilon}T^{\varepsilon+\frac{1}{10}}\sup_{0<t<T}\|w\|_{E^s}\Big[1+T^{1-\frac{4\theta}{5}}(\|w\|_{L^{\frac{5}{2\theta}}_{T}L^{\infty}_{xy}}+\|D_{y}^{3\epsilon}w\|_{L^{\frac{5}{2\theta}}_{T}L^{\infty}_{xy}})^2\Big].$$
\end{corollary}
We proceed like in Corollary 5.2 in \cite{KenigZiesler}, therefore we omit the proof. 

\begin{lemma} \label{aprioriestimates2}
    Let $T,w$ be as in Lemma \ref{aprioriestimate}. Then 
    $$\|w\|_{L^{\frac{5}{2\theta}}_{T}L^{\frac{2}{1-\theta}}_{xy}}\leq C(\sup_{0<t<T}\|\langle D_x\rangle^{\frac{2\theta}{5}}w\|_{L^{2}_{xy}}+T^{\frac{1}{2}}\|\langle D_x\rangle^{-\frac{1}{2}}F\|_{L^{2}_{T}L^{2}_{xy}}).$$
\end{lemma}

This linear estimate follows Lemma 5.4 in \cite{KenigZiesler}, by applying Lemma \ref{StrichartzEstimates} to Duhamel's formula, i.e. 
$$\|W(t)\phi\|_{L^{\frac{5}{2\theta}}_TL^{\frac{2}{1-\theta}}_{xy}}\leq \|\phi\|_{L^2_{xy}}.$$ We proceed with a Littlewood-Paley decomposition as in the above-mentioned Lemma and we finish the proof.

\begin{corollary} \label{littlewoodpaley}
    Let $w$ be as in Lemma \ref{aprioriestimate}. Let $\varepsilon>0.$ Then 
    $$\|w\|_{L^{\frac{5}{2\theta}}_{T}L^{\infty}_{xy}}\leq C_{\varepsilon}[\sup_{0<t<T}\|\langle D_x\rangle^{1-\theta+\frac{2\theta}{5}+\varepsilon}\langle D_{x}^{-1}D_{y}\rangle^{\frac{1-\theta}{2}+\varepsilon}w\|_{L^{2}_{xy}}+T^{\frac{1}{2}}\|\langle D_x\rangle^{\frac{1}{2}-\theta+\varepsilon}\langle D_{x}^{-1}D_{y}\rangle^{\frac{1-\theta}{2}+\varepsilon}F\|_{L^{2}_{T}L^{2}_{xy}}].$$
\end{corollary}
\begin{proof}
    First, we prove the following estimate:
    $$\|w\|_{L^{\infty}_{xy}}\leq C\|\langle D_x\rangle^{1-\theta+\varepsilon}\langle D_{x}^{-1}D_{y}\rangle^{\frac{1-\theta}{2}+\varepsilon}w\|_{L^{\frac{2}{1-\theta}}_{xy}}$$
    We denote $L=\langle D_x\rangle^{1-\theta+\varepsilon}\langle D_{x}^{-1}D_{y}\rangle^{\frac{1-\theta}{2}+\varepsilon}$ and let $u_{k,l}$ be the Littlewood-Paley projection of $u$ to frequencies $\xi\approx 2^k$ and $\mu\approx 2^l.$ If we define $m(\xi, \mu)=\langle \xi\rangle^{1-\theta+\varepsilon}\langle \frac{\mu}{\xi}\rangle^{\frac{1-\theta}{2}+\varepsilon},$ therefore we observe that 
    $Lu=\check{m\hat{u}}.$ By the Littlewood-Paley theory we know that 
    $$\|Lu_{k,l}\|_{L^{p}_{xy}}\approx 2^{\max\{0,k\}(1-\theta+\varepsilon)+\max\{0,l-k\}(\frac{1-\theta}{2}+\varepsilon)}\|u_{k,l}\|_{L^{p}_{xy}}.$$
    By Bernstein-Sobolev inequality, we also have that 
    $$\|w\|_{L^{\infty}_{xy}}\lesssim (2^{k+l})^{\frac{1-\theta}{2}}\|w\|_{L^{\frac{2}{1-\theta}}_{xy}}.$$
    Using that $(k+l)\frac{1-\theta}{2}=k(1-\theta)+(l-k)\frac{1-\theta}{2},$ we have 
    $$\|u_{k,l}\|_{L^{\infty}_{xy}}\leq 2^{\min(0,k)(1-\theta)-\varepsilon\max(0,k)+\min(0,l-k)\frac{1-\theta}{2}-\varepsilon\max(0,l-k)}\sup_{k,l}\|Lu_{k,l}\|_{L^{\frac{2}{1-\theta}}_{xy}},$$
    hence, by summing over $k$ and $l$, we get 
    $$\|u\|_{L^{\infty}_{xy}}\leq \sum_{k,l \in \mathbb{Z}}\|u_{k,l}\|_{L^{\infty}_{xy}}\leq \sup_{k,l}\|Lu_{k,l}\|_{L^{\frac{2}{1-\theta}}_{xy}}\leq \|Lu\|_{L^{\frac{2}{1-\theta}}_{xy}}$$
    and the estimate is proved. We conclude using this estimate combined with Lemma \ref{aprioriestimates2}.
    
\end{proof}

\begin{lemma}\label{mainbound}
    Let $T,w$ be as in Lemma \ref{aprioriestimate} and let $F=-w^2\partial_xw.$ Then, for all $\varepsilon>0, s>\frac{3}{2},$
   $$\|w\|_{L^{\frac{5}{2\theta}}L^{\infty}_{xy}}+\|D_{y}^{3\varepsilon}w\|_{L^{\frac{5}{2\theta}}_{T}L^{\infty}_{xy}}$$
    $$\leq C[\sup_{0<t<T}\|w\|_{E^s}+T^{1-\frac{4\theta}{5}}(\sup_{0<t<T}\|w\|_{E^s}(\|w\|_{L^{\frac{5}{2\theta}}_{T}L^{\infty}_{xy}}^{2}+\|D_{y}^{3\varepsilon}w\|_{L^{\frac{5}{2\theta}}_{T}L^{\infty}_{xy}}^{2})+\sup_{0<t<T}\|w\|_{E^s}^{3})].$$
\end{lemma}

\begin{proof}
    First, for a fixed $\theta \in (0,\frac{5}{8}),$ we bound the quantity $\|u\|_{L^{\frac{5}{2\theta}}_{T}L^{\infty}_{xy}}.$ From Lemma \ref{littlewoodpaley}, we obtain 
    $$\|u\|_{L^{\frac{5}{2\theta}}_{T}L^{\infty}_{xy}}\leq [\sup_{0<t<T}\|\langle D_x\rangle^{1-\frac{3\theta}{5}+\varepsilon}\langle D_{x}^{-1}D_{y}\rangle^{\frac{1-\theta}{2}+\varepsilon}w\|_{L^{2}_{xy}}+T^{\frac{1}{2}}\|D_x\langle D_x\rangle^{\frac{1}{2}-\theta+\varepsilon}\langle D_{x}^{-1}D_{y}\rangle^{\frac{1-\theta}{2}+\varepsilon}w^3\|_{L^{2}_{T}L^{2}_{xy}}]$$

    First, we notice that 
    $$\|\langle D_x\rangle^{1-\frac{3\theta}{5}+\varepsilon}\langle D_{x}^{-1}D_{y}\rangle^{\frac{1-\theta}{2}+\varepsilon}w\|_{L^{2}_{xy}}\leq \|w\|_{E^s}$$
    if $s\geq \frac{1-\frac{3\theta}{5}+\varepsilon}{\frac{1+\theta}{2}-\varepsilon}.$

    We have that 
    $$\|D_x\langle D_x\rangle^{\frac{1}{2}-\theta+\varepsilon}\langle D_{x}^{-1}D_{y}\rangle^{\frac{1-\theta}{2}+\varepsilon}w^3\|_{L^{2}_{T}L^{2}_{xy}}\lesssim \|\Delta_0w^3\|_{L^{2}_{T}L^{2}_{xy}}+\|D_{y}^{\frac{1-\theta}{2}+\varepsilon}\Delta_0w^3\|_{L^{2}_{T}L^{2}_{xy}}$$
    $$+\|\langle D_x\rangle^{\frac{3}{2}-\theta+\varepsilon}\langle D_{x}^{-1}D_{y}\rangle^{\frac{1-\theta}{2}+\varepsilon}(1-\Delta_0)w^3\|_{L^{2}_{T}L^{2}_{xy}},$$
    where we used the inequality $$1_{\{|\xi|\leq 1\}}|\xi|^2(1+|\xi|^2)^{\frac{1}{2}-\theta+\varepsilon}(1+\Big|\frac{\mu}{\xi}\Big|^2)^{\frac{1-\theta}{2}+\varepsilon}\lesssim (1+|\mu|^{1-\theta+2\varepsilon})1_{\{|\xi|\leq 1\}},$$ for $\xi \in \mathbb{R}\setminus\{0\}, \mu\in \mathbb{R}.$
    For the first term, we bound it as 
    $$\|\Delta_0w^3\|_{L^{2}_{T}L^{2}_{xy}}\lesssim \|w\|_{L^{2}_{T}L^{\infty}_{xy}}^{2}\sup_{0<t<T}\|w\|_{E^s}T^{\frac{5-8\theta}{10}}$$
    Using Kato's inequalities, which says that for $0<\alpha<1, 1<p<\infty,$ we have 
    \begin{equation}\label{Kato}
        \|D^{\alpha}(fg)-fD^{\alpha}g-gD^{\alpha}f\|_{L^p(\mathbb{R})}\leq C\|g\|_{L^{\infty}(\mathbb{R})}\|D^{\alpha}f\|_{L^p(\mathbb{R})}.
    \end{equation}

    we can bound the second term as
    $$\|D_{y}^{\frac{1-\theta}{2}+\varepsilon}\Delta_0w^3\|_{L^{2}_{T}L^{2}_{xy}}\lesssim \|w\|_{L^{\frac{5}{2\theta}}_{T}L^{\infty}_{xy}}^{2}\|D_{y}^{\frac{1-\theta}{2}+\varepsilon}w\|_{L^{\infty}_{T}L^{2}_{xy}}T^{\frac{5-8\theta}{10}}\lesssim \|w\|_{L^{\frac{5}{2\theta}}_{T}L^{\infty}_{xy}}^{2}\sup_{0<t<T}\|w\|_{E^s}T^{\frac{5-8\theta}{10}},$$
    since $s\geq \frac{1-\theta}{2}+\varepsilon.$

    The third term is more complicated and we will estimate it in several steps. First, observe that 
    $$\|\langle D_x\rangle^{\frac{3}{2}-\theta+\varepsilon}\langle D_{x}^{-1}D_{y}\rangle^{\frac{1-\theta}{2}+\varepsilon}(1-\Delta_0)w^3\|_{L^{2}_{T}L^{2}_{xy}}$$
    $$\lesssim \|\langle D_x\rangle^{1-\frac{\theta}{2}+\varepsilon}\langle D_{y}\rangle^{\frac{1-\theta}{2}+\varepsilon}(1-\Delta_0)(1-\Delta^{2}_{0})w^3\|_{L^{2}_{T}L^{2}_{xy}}+\|\langle D_x\rangle^{\frac{3}{2}-\theta+\varepsilon}(1-\Delta_0)w^3\|_{L^{2}_{T}L^{2}_{xy}}$$
    $$=A_1+A_2,$$
    where we used the inequality $$(1+|\xi|^2)^{\frac{3}{2}-\theta+\varepsilon}\Big(1+\Big|\frac{\mu}{\xi}\Big|^2\Big)^{\frac{1-\theta}{2}+\varepsilon}1_{\{|\xi|\geq 1\}}\leq (1+|\xi|^2)^{1-\frac{\theta}{2}}|\mu|^{1-\theta+2\varepsilon}1_{\{|\xi|\geq 1\}}1_{\{|\mu|\geq 1\}}+(1+|\xi|^2)^{\frac{3}{2}-\theta+\varepsilon}1_{\{|\xi|\geq 1\}}.$$

    We start with $A_1.$ Using the dyadic decomposition, we obtain that 
    $$\|\langle D_x\rangle^{1-\frac{\theta}{2}+\varepsilon}\langle D_{y}\rangle^{\frac{1-\theta}{2}+\varepsilon}(1-\Delta_0)(1-\Delta^{2}_{0})w^3\|_{L^{2}_{xy}}\lesssim$$
    $$\lesssim \|w\|_{L^{\infty}_{xy}}\|\langle D_x\rangle^{1-\frac{\theta}{2}}w\|_{L^{\frac{2}{1-\theta}}_{xy}}\|D_{y}^{\frac{1-\theta}{2}+\varepsilon}w\|_{L^{\frac{2}{\theta}}_{xy}}+\|w\|_{L^{\infty}_{xy}}^2\|\langle D_x\rangle^{1-\frac{\theta}{2}+\varepsilon} D_{y}^{\frac{1-\theta}{2}+\varepsilon}w\|_{L^{2}_{xy}}=A_{1,1}+A_{1,2}.$$
    For $A_{1,2}\leq \|w\|_{L^{\infty}_{xy}}^2\|w\|_{E^s}$ for $s\geq\frac{\frac{3}{2}-\theta+2\varepsilon}{\frac{1+\theta}{2}-\varepsilon}.$

    For the term $A_{1,1}$, we divide the estimates in several parts. First, 
    $$\|D_{y}^{\frac{1-\theta}{2}+\varepsilon}w\|_{L^{\frac{2}{\theta}}_{xy}}\lesssim \||\mu|^{\frac{1-\theta}{2}+\varepsilon}\hat{w}\|_{L^{\frac{2}{2-\theta}}_{\xi\mu}}$$
    $$\leq \norm{\frac{|\mu|^{\frac{1-\theta}{2}+\varepsilon}}{(1+\xi^2)^{\frac{a}{2}}(1+\frac{\mu^2}{\xi^2})^{\frac{b}{2}}}}_{L^{\frac{2}{1-\theta}}_{\xi\mu}}\norm{(1+\xi^2)^{\frac{a}{2}}\Big(1+\frac{\mu^2}{\xi^2}\Big)^{\frac{b}{2}}\hat{w}}_{L^{2}_{\xi\mu}}$$
    and the first factor is bounded if $a>\frac{3}{2}(1-\theta)+\varepsilon, b>1-\theta+\varepsilon.$ Hence, 
    $$\|D_{y}^{\frac{1-\theta}{2}+\varepsilon}w\|_{L^{\frac{2}{\theta}}_{xy}}\lesssim \norm{\langle D_x\rangle^{\frac{3}{2}(1-\theta)+\varepsilon}\langle D_{x}^{-1}D_{y}\rangle^{1-\theta+\varepsilon}w}_{L^{2}_{xy}}\lesssim \|w\|_{E^s},$$
    for $s\geq \frac{\frac{3}{2}(1-\theta)+\varepsilon}{\theta-\varepsilon}.$

    For the remaining term in $A_{1,1}$ we proceed by a dyadic decomposition. We observe that 
    $$\|\langle D_x\rangle^{1-\frac{\theta}{2}}\Delta_0w\|_{L^{\frac{2}{1-\theta}}_{xy}}\lesssim \|\langle D_x\rangle^{1-\frac{\theta}{2}}\Delta_0w\|_{L^{2}_{xy}}^{1-\theta}\|\langle D_x\rangle^{1-\frac{\theta}{2}}\Delta_0w\|_{L^{\infty}_{xy}}^{\theta}\leq \|w\|_{E^s}^{1-\theta}\|w\|_{L^{\infty}_{xy}}^{\theta}\|\langle D_x\rangle^{1-\frac{\theta}{2}}\Delta_0\check{\eta}\|_{L^{1}_{x}}^{\theta}$$
    $$\leq \|w\|_{E^s}^{1-\theta}\|w\|_{L^{\infty}_{xy}}^{\theta}.$$
    For higher frequencies, we apply the Gagliardo-Nirenberg-Brezis-Mironescu inequality to get for $k\geq 1,$
     $$\|\langle D_x\rangle^{1-\frac{\theta}{2}}\Delta_kw\|_{L^{\frac{2}{1-\theta}}_{x}}\leq C\|\langle D_x\rangle^{\frac{2-\theta}{2(1-\theta)}}\Delta_kw\|_{L^{2}_{x}}^{1-\theta}\|\Delta_kw\|_{L^{\infty}_{x}}^{\theta},$$
    which implies that 
    $$\|\langle D_x\rangle^{1-\frac{\theta}{2}}\Delta_kw\|_{L^{\frac{2}{1-\theta}}_{xy}}\leq C\|\langle D_x\rangle^{\frac{2-\theta}{2(1-\theta)}}\Delta_kw\|_{L^{2}_{xy}}^{1-\theta}\|\Delta_kw\|_{L^{\infty}_{xy}}^{\theta}\leq C2^{-k\varepsilon}\|\langle D_x\rangle^{\frac{2-\theta}{2(1-\theta)}+\varepsilon}\Delta_kw\|_{L^{2}_{xy}}^{1-\theta}\|\Delta_kw\|_{L^{\infty}_{xy}}^{\theta}$$
    $$\leq C2^{-k\varepsilon}\|w\|_{E^s}^{1-\theta}\|w\|_{L^{\infty}_{xy}}^{\theta}\|\check{\eta}\|_{L^{1}_{x}}^{\theta}\leq C2^{-k\varepsilon}\|w\|_{E^s}^{1-\theta}\|w\|_{L^{\infty}_{xy}}^{\theta},$$
    for $s\geq \frac{2-\theta}{2(1-\theta)}+\varepsilon.$ Summing over $k,$ we obtain 
    $$\|\langle D_x\rangle^{1-\frac{\theta}{2}}w\|_{L^{\frac{2}{1-\theta}}_{xy}}\leq C\|w\|_{E^s}^{1-\theta}\|w\|_{L^{\infty}_{xy}}^{\theta}.$$
    We conclude that $$\|\langle D_x\rangle^{1-\frac{\theta}{2}+\varepsilon}\langle D_{y}\rangle^{\frac{1-\theta}{2}+\varepsilon}(1-\Delta_0)(1-\Delta^{2}_{0})w^3\|_{L^{2}_{xy}}\leq C\{\|w\|_{E^s}^{2-\theta}\|w\|_{L^{\infty}_{xy}}^{1+\theta}+\|w\|_{E^s}\|w\|_{L^{\infty}_{xy}}^{2}\},$$
    hence by H\"{o}lder's inequality in time, together with the AM-GM inequality we have
    $$A_1\leq \sup_{0<t<T}\|w\|_{E^s}\|w\|_{L^{\frac{5}{2\theta}}_{T}L^{\infty}_{xy}}^{2}T^{\frac{5-8\theta}{10}}+\sup_{0<t<T}\|w\|_{E^s}^{2-\theta}\|w\|_{L^{\frac{5}{2\theta}}_{T}L^{\infty}_{xy}}^{1+\theta}T^{\frac{5-4\theta(1+\theta)}{10}}$$
    $$\lesssim T^{\frac{5-8\theta}{10}}(\sup_{0<t<T}\|w\|_{E^s}\|w\|_{L^{\frac{5}{2\theta}}_{T}L^{\infty}_{xy}}^{2}+\sup_{0<t<T}\|w\|_{E^s}^{3}).$$

    For the term $A_2$ we start by applying Kato's inequality \eqref{Kato} to obtain 
    $$\|\langle D_x\rangle^{\frac{3}{2}-\theta+\varepsilon}w^3\|_{L^{2}_{xy}}\leq \|w\|_{L^{\infty}_{xy}}^2\|\langle D_x \rangle^{\frac{3}{2}-\theta+\varepsilon}w\|_{L^{2}_{xy}},$$
    and since $s>\frac{3}{2}-\theta+\varepsilon,$ this implies that 
    $$A_2\leq C\|w\|_{L^{\frac{5}{2\theta}}_{T}L^{\infty}_{xy}}^2\sup_{0<t<T}\|w\|_{E^s}T^{\frac{5-8\theta}{10}}.$$

    We conclude that if $s>\min_{\theta\in (0,\frac{5}{8})}\max\{\frac{3-2\theta}{1+\theta},\frac{3-3\theta}{2\theta},\frac{2-\theta}{2-2\theta}\}$
    $$\|w\|_{L^{\frac{5}{2\theta}}L^{\infty}_{xy}}\lesssim \sup_{0<t<T}\|w\|_{E^s}+T^{1-\frac{4\theta}{5}}(\sup_{0<t<T}\|w\|_{E^s}\|w\|_{L^{\frac{5}{2\theta}}_{T}L^{\infty}_{xy}}^{2}+\sup_{0<t<T}\|w\|_{E^s}^{3}). $$

    For the term $\|D_{y}^{3\varepsilon}w\|_{L^{\frac{5}{2\theta}}_{T}L^{\infty}_{xy}}$ we proceed the same way as above, the only difference will be for the equivalent term 
    $$\|\langle D_x\rangle^{\frac{3}{2}-\theta+\varepsilon}D_{y}^{3\varepsilon}w\|_{L^{2}_{T}L^{2}_{xy}}\leq C\|w\|_{L^{\frac{5}{2\theta}}_{T}L^{\infty}_{xy}}(\|w\|_{L^{\frac{5}{2\theta}}_{T}L^{\infty}_{xy}}+\|D_{y}^{3\varepsilon}w\|_{L^{\frac{5}{2\theta}}_{T}L^{\infty}_{xy}})\sup_{0<t<T}\|w\|_{E^s}T^{\frac{5-8\theta}{10}},$$
    which implies that for $s>\min_{\theta\in [0,\frac{5}{8})}\max\{\frac{3-2\theta}{1+\theta},\frac{3-3\theta}{2\theta},\frac{2-\theta}{2-2\theta}\}$
    $$\|w\|_{L^{\frac{5}{2\theta}}L^{\infty}_{xy}}+\|D_{y}^{3\varepsilon}w\|_{L^{\frac{5}{2\theta}}_{T}L^{\infty}_{xy}}$$
    $$\leq C\{\sup_{0<t<T}\|w\|_{E^s}+T^{1-\frac{4\theta}{5}}[\sup_{0<t<T}\|w\|_{E^s}(\|w\|_{L^{\frac{5}{2\theta}}_{T}L^{\infty}_{xy}}+\|D_{y}^{3\varepsilon}w\|_{L^{\frac{5}{2\theta}}_{T}L^{\infty}_{xy}})^{2}+\sup_{0<t<T}\|w\|_{E^s}^{3}]\}.$$
    We conclude by noting that $\min_{\theta\in (0,\frac{5}{8})}\max\{\frac{3-2\theta}{1+\theta},\frac{3-3\theta}{2\theta},\frac{2-\theta}{2-2\theta}\}=\frac{3}{2}$ attained for $\theta=\frac{1}{2}.$
    
    \end{proof}

    \begin{proposition} \label{claim}
        We suppose that $u$ is a solution to \eqref{fifthKPI2} on $[0,T_0], T_0<1$ such that $\sup_{0<t<T}\|u(t)\|_{E^s}\leq 2\|u_0\|_{E^s}$ and $T_0<\frac{1}{A\|u_0\|_{E^s}^{4}(1+\|u_0\|_{E^s}^{2})^{2}},$ for $A$ large. 
        Then 
        $$\|u\|_{L^{2}_{T_0}L^{\infty}_{xy}}+\|\partial_x u\|_{L^{2}_{T_0}L^{\infty}_{xy}}\leq C_{\varepsilon}T_{0}^{\varepsilon}(1+\|u_0\|_{E^s}+\|u_0\|_{E^s}^{3}).$$
    \end{proposition}
    \begin{proof}
        Define $f(T_0)=\|u\|_{L^{5}_{T_0}L^{\infty}_{xy}}+\|\partial_x u\|_{L^{5}_{T_0}L^{\infty}_{xy}}.$ From Lemma \ref{mainbound} and the first assumption of the proposition, we find that 
        $$f(T_0)\leq 8C(\|u_0\|_{E^s}+\|u_0\|_{E^s}^{3})+2CT_0^{\frac{3}{5}}\|u_0\|_{E^s}f(T_0)^2.$$
        If $128C^2<A^{\frac{1}{2}}$ for $A$ sufficiently large and using the assumption that $T_0<\frac{1}{A\|u_0\|_{E^s}^{4}(1+\|u_0\|_{E^s}^{2})^{2}},$ then $f(T_0)\leq 2B=16C(\|u_0\|_{E^s}+\|u_0\|_{E^s}^{3}).$ 
        By Corollary \ref{boundforL2T} for $\theta=\frac{1}{2},$ we obtain 
        $$\|u\|_{L^{2}_{T_0}L^{\infty}_{xy}}+\|\partial_x u\|_{L^{2}_{T_0}L^{\infty}_{xy}}\leq C_{\varepsilon}T_{0}^{\varepsilon+\frac{1}{10}}\sup_{0<t<T_{0}}\|u_0\|_{E^s}[1+64CT_{0}^{\frac{3}{5}}(\|u_0\|_{E^s}+\| u_0\|_{E^s}^{3})^2].$$
        
    \end{proof}
         Define 
        $$T_1=\min\Big\{1, \frac{1}{A\|u_0\|_{E^s}^{4}(1+\|u_0\|_{E^s}^{2})^2}, \Big[\frac{3}{4CC_{\varepsilon}^{2}(1+\|u_0\|_{E^s}+\|u_0\|_{E^s}^{3})^2\|u_0\|_{E^s}^{2}}\Big]^{\frac{1}{2(\varepsilon+\frac{1}{10})}}\Big\}.$$ 
        Also, denote $\tilde{T}_{1}=\inf\{t>0: \|u(t)\|_{E^s}\geq 2\|u_0\|_{E^s}\}.$
   \begin{proposition}\label{welposedness}
       We suppose that $u$ is a solution to \ref{fifthKPI2} on $[0,\tilde{T}_1].$ Then: 
       \begin{itemize}
           \item[a)] $\sup_{0<t<\tilde{T}_1}\|u(t)\|_{E^s}\leq 2\|u_0\|_{E^s}$
            \item[b)] $\|u\|_{L^{2}_{\tilde{T}_1}L^{\infty}_{xy}}+\|\partial_x u\|_{L^{2}_{\tilde{T}_1}L^{\infty}_{xy}}\leq C_{\varepsilon}\tilde{T}_{1}^{\varepsilon}(1+\|u_0\|_{E^s}+\|u_0\|_{E^s}^{3}).$
       \end{itemize}
   \end{proposition}     
        \begin{proof}
          This follows from \cite{KenigZiesler} combining Lemma \ref{energyestimates} with Proposition \ref{claim}, by a contradiction argument. 
        \end{proof}
        Using Proposition \ref{mainproposition}, we conclude that \eqref{fifthKPI2} is locally well-posed in $E^s$ for $s>\frac{3}{2}.$
       
\subsection{Global Well-Posedness for the $L^2$ critical case of the Fifth Order mKP-I} 

We start with the sharp anisotropic Sobolev inequality stated in \cite{EsfahaniLevandovsky}. 

\begin{lemma} 
    Let $\varphi\in G_1$, where $G_1$ is defined in \ref{defgroundstate}. For $0<p<4,$ we have the following: 
    $$\|u\|_{L^{p+2}(\mathbb{R}^2)}^{p+2}\leq \frac{p+2}{2}\Big(\frac{4}{p}\Big)^{\frac{p}{2}}\|\varphi\|_{L^2(\mathbb{R}^2)}^{-p}\|u\|_{L^2(\mathbb{R}^2)}^{2}\|\partial_{xx}u\|_{L^2(\mathbb{R}^2)}^{\frac{p}{2}}\|\partial_{x}^{-1}\partial_{y}u\|_{L^2(\mathbb{R}^2)}^{\frac{p}{2}},$$
    and the inequality is sharp with equality for $u=\varphi.$
\end{lemma}

The proof is the same as the sharp anisotropic Sobolev inequality showed in \cite{ChenFengLiu}. For $p=2,$ we have that 
\begin{equation}\label{anisotropic}
    \|u\|_{L^{4}(\mathbb{R}^2)}^{4}\leq 4\|\varphi\|_{L^2(\mathbb{R}^2)}^{-2}\|u\|_{L^2(\mathbb{R}^2)}^{2}\|\partial_{xx}u\|_{L^2(\mathbb{R}^2)}\|\partial_{x}^{-1}\partial_{y}u\|_{L^2(\mathbb{R}^2)}.
\end{equation}

Using \eqref{anisotropic}, we notice that 
$$\|u(t)\|_{E^2}^{2}\leq \|u(t)\|_{L^2}^{2}+2E(u(t))+\frac{1}{2}\|u(t)\|_{L^4}^{4}$$
$$\leq \|u_0\|_{L^2}^{2}+2E(u_0)+\frac{\|u_0\|_{L^2}^{2}}{\|\varphi\|_{L^2}^{2}}(\|\partial_{xx}u(t)\|_{L^2}^{2}+\|\partial_{x}^{-1}\partial_{y}u(t)\|_{L^2}^{2})$$
and if, $\|u_0\|_{L^2}<\|\varphi\|_{L^2},$ it implies that 
$$\|u(t)\|_{E^2}^{2}\leq \|u_0\|_{L^2}^{2}+\frac{2}{1-\frac{\|u_0\|_{L^2}}{\|\varphi\|_{L^2}^{2}}}E(u_0),$$
and we conclude in this case that equation \eqref{fifthKPI2} is globally well-posed in $E^2.$

\section{Transverse Blow-Up for the Generalized fifth order KP-I}\label{BU}

Recall that $L_{1}(u)=E(u)+V(u)$ with $E(u)=\frac{1}{2}\int_{\mathbb{R}^2}(\partial_{x}^{2}u)^2+(\partial_{x}^{-1}\partial_{y}u)^2dxdy-\frac{1}{p+2}\int_{\mathbb{R}^2}u^{p+2}$ and $V(u)=\frac{1}{2}\int_{\mathbb{R}^2}u^2.$ 
We define $I(u)=\int_{\mathbb{R}^2}[u^2+(\partial_{x}^{2}u)^2+(\partial_{x}^{-1}\partial_{y}u)^2-u^{p+2}]dxdy.$

We have the following minimization formulation of the ground states: 
\begin{theorem}\label{minimization}[Theorem 2.4, \cite{EsfahaniLevandovsky}] 
The following are equivalent. 
\begin{itemize}
    \item[a)]  $u$ is a ground state, i.e. $u\in G_1$.
    \item[b)]  $I(u)=0$ and $d_1=L_{1}(u)=\inf\{L(w)|w\in E^2\setminus \{0\}, I(w)=0\}.$
\end{itemize}
\end{theorem}
We denote $J(t)=\int_{\mathbb{R}^2}y^2u(x,y,t)^2dxdy.$ From \cite{Saut}, we have that if $u_0\in E^2$ with $yu_0\in L^{2}(\mathbb{R}^2)$, then for a solution of equation \eqref{fifthKPI}, $u(t)\in C([0,T],E^2)$ with initial data $u_0,$ we have $yu \in L^2(\mathbb{R}^2),$ hence $J(t)$ is well-defined. Drawing inspiration from the virial identity that first appeared in \cite{FalkovitchTuritsyn}, \cite{Saut}, we state the virial identity associated to \eqref{fifthKPI}: 
\begin{lemma}\label{virial}[Theorem 3.4, \cite{EsfahaniLevandovsky}]
We have 
$$\frac{1}{8}\frac{d^2}{dt^2}J(t)=\int_{\mathbb{R}^2}(\partial_{x}^{-1}\partial_{y}u)^2-\frac{p}{2(p+2)}\int_{\mathbb{R}^2}u^{p+2}.$$   
\end{lemma}

We now prove a minimization property that resembles Lemma 3.2 from \cite{Esfahani}: 
\begin{lemma} \label{minimization2}
    For $u \in E^2$, we define 
    $$K(u)=\int_{\mathbb{R}^2}[u^2+(\partial_{x}^{2}u)^2]-\frac{p+4}{2(p+2)}\int_{\mathbb{R}^2}u^{p+2},$$
    and thus $\mathcal{D}=\{u\in E^2\setminus\{0\},\mbox{ } I(u)<0,\mbox{ }K(u)=0\}.$
    Then we have $\inf_{w \in \mathcal{D}} L_{1}(w)\geq d_1.$ Moreover, if $\varphi \in G_1,$ then $K(\varphi)=0.$ 
\end{lemma}
\begin{proof}
    For any $u\in \mathcal{D},$ we will find $w\in E^2$ with $I(w)=0$ such that $L_1(u)\geq L_1(w).$ By Theorem \ref{minimization}, it will imply the conclusion of the Lemma. We take now any $u \in \mathcal{D}$, hence $I(u)<0, K(u)=0.$ For $\lambda>0,$ we denote $u_{\lambda}(x,y)=\lambda u(x,\lambda^{-2}y).$ We notice that 
    $$I(u_{\lambda})=\lambda^4\int_{\mathbb{R}^2}[u^2+(\partial_{xx}u)^2]+\int_{\mathbb{R}^2}(\partial_{x}^{-1}\partial_{y}u)^2-\lambda^{p+4}\int_{\mathbb{R}^2}u^{p+2}.$$
    Hence, as $\lambda\rightarrow 1^{-},$ then $I(u_{\lambda})\rightarrow I(u)<0$ and as $\lambda \rightarrow 0^{+},$ then $I(u_{\lambda})\rightarrow \int_{\mathbb{R}^2}(\partial_{x}^{-1}\partial_{y}u)^2>0$ (we observe that $\int_{\mathbb{R}^2}(\partial_{x}^{-1}\partial_{y}u)^2>0,$ by \eqref{anisotropic} and the fact that $u\not\equiv 0$). By continuity, we have that there exists $\rho \in (0,1),$ such that $I(u_{\rho})=0.$  
    We see that 
    $$L_1(u_{\lambda})=\frac{\lambda^4}{2}\int_{\mathbb{R}^2}[u^2+(\partial_{x}^{2}u)^2]+\frac{1}{2}\int_{\mathbb{R}^2}(\partial_{x}^{-1}\partial_{y}u)^{2}-\frac{\lambda^{p+4}}{p+2}\int_{\mathbb{R}^2}u^{p+2},$$
    hence $$L_1(u)-L_1(u_{\lambda})=\frac{1-\lambda^4}{2}\int_{\mathbb{R}^2}[u^2+(\partial_{x}^{2}u)^2]-\frac{1-\lambda^{p+4}}{p+2}\int_{\mathbb{R}^2}u^{p+2}.$$
    Using $K(u)=0,$ we obtain that 
    $$L_1(u)-L_1(u_{\lambda})=f(\lambda)\int_{\mathbb{R}^2}u^2+g(\lambda)\int_{\mathbb{R}^2}(\partial_{x}^{2}u)^2$$
    where $f(\lambda)=g(\lambda)=\frac{1}{2}(1-\lambda^{4})-\frac{2}{p+4}(1-\lambda^{p+4}).$ We observe that $f(0)>0, f(1)=0$ and $f'(\lambda)=2(\lambda^p-1)\lambda^3<0$ for $0<\lambda<1,$ thus $f(\lambda)>0$ for $\lambda \in (0,1).$ 
    We conclude that $L_1(u)>L_1(u_{\lambda})$ for $\lambda \in (0,1),$ in particular $L_1(u)>L_1(u_{\rho}).$
    Moreover, from Lemma \ref{pohozaev} we get that $K(\varphi)=0$ for $\varphi \in G_1.$
\end{proof}

    We now define an invariant set that will help us prove transverse blow-up. Denote 
    $\mathcal{J}=\{ u\in X_s: L_1(u)<d_1, I(u)<0, K(u)>0\}.$
    \begin{lemma} \label{invariance}
        Let $u_0 \in X_s.$ Then if $u(t)\in C([0,T), X_s)$ is the solution of \eqref{fifthKPI} with initial data $u_0$, with $V(u(t))=V(u_0)$ and $E(u(t))=E(u_0),$ for $0\leq t<T.$ Then, if $u_0 \in \mathcal{J},$ we have $u(t)\in \mathcal{J}$ for $0\leq t<T.$     
    \end{lemma}
    \begin{proof}
        Let $u_0 \in \mathcal{J}.$ By the conservation laws, we get that $L_1(u(t))=L_1(u_0)<d_1.$ Suppose, by contradiction, that there exists $t_0>0$ such that $I(u(t_0))>0.$ By continuity, we get that there exists $t_1$ such that $I(u(t_1))=0.$ By Theorem \ref{minimization}, we obtain that $L_1(u(t_1))\geq d_1,$ contradiction. 

    Again, by contradiction, suppose that there exists $t_0>0$ such that $K(u(t_0))<0.$ By continuity, we get that there exists $t_1$ such that $K(u(t_1))=0.$ By Lemma \ref{minimization2}, we obtain that $L_1(u(t_1)\geq d_1,$ contradiction.   
    \end{proof}
        We now prove the blow-up criterion for a certain set of initial data.
    \begin{lemma}\label{blowup}
        Let $u(t)$ be a solution of \eqref{fifthKPI}, with $p>0, p=\frac{m}{n}$ and $m$ is even, $n$ is odd, in $C([0,T),X_s)$ for $s\geq 5,$ having initial data $u_0\in X_s$ with $u_0\in \mathcal{J}$ and $yu_0 \in L^2(\mathbb{R}^2).$ Then there exists a blow-up time $T_0<\infty$ such that 
        $$\lim_{t\rightarrow T_{0}^{-}}\|\partial_{y}u\|_{L^{2}(\mathbb{R}^2)}=\infty.$$
    \end{lemma} 
    \begin{proof}
        From Lemma \ref{virial} and Lemma \ref{invariance}, we observe that 
        $$\frac{1}{8}\frac{d^2}{dt^2}J(t)=-K(u(t))+I(u(t))<0.$$
        Denote, for any $\lambda>0,$ $u_{\lambda}(x,y)=\lambda^2u(\lambda x, \lambda^3y),$ and we notice that  
        $$I(u_{\lambda})\rightarrow \|u(t)\|_{L^2}^2=\|u_0\|_{L^2}^2>0, \mbox{ as }\lambda\rightarrow 0,$$
        and 
        $$I(u_{\lambda})\rightarrow I(u)<0, \mbox{ as }\lambda\rightarrow 1.$$
        Therefore there exists $\mu \in (0,1)$ such that $I(u_{\mu})=0.$
        By simple computations, we have 
        $$L_1(u)-L_1(u_{\mu})=\frac{1-\mu^4}{2}\int_{\mathbb{R}^2}(\partial_{x}^{2}u)^2+(\partial_{x}^{-1}\partial_{y}u)^2-\frac{1-\mu^{2p}}{p+2}\int_{\mathbb{R}^2}u^{p+2}$$
        $$\geq \frac{1-\mu^4}{2}\int_{\mathbb{R}^2}(\partial_{x}^{2}u)^2+(\partial_{x}^{-1}\partial_{y}u)^2-\frac{1-\mu^{2p}}{2}\int_{\mathbb{R}^2}u^{p+2}=I(u)-I(u_{\mu})=I(u).$$ 
        By conservation laws we have that $L_1(u(t))=L_1(u_0)$ and by Theorem \ref{minimization} since $I(u_{\mu})=0,$ it implies that $L_1(u_{\mu})\geq d_1.$ Thus, 
        $$0>L_1(u_0)-d_1>I(u(t))>\frac{d^2}{dt^2}J(t),$$ 
        which implies that there exists $T_0<\infty$ such that $\lim_{t\rightarrow T_{0}^{-}}J(t)=0.$ By the Weyl-Heisenberg inequality, i.e. 
        $$\|u_0\|_{L^2}^2=\|u(t)\|_{L^2}^2\lesssim J(t)\|\partial_{y}u(t)\|_{L^2},$$
        it yields $\lim_{t\rightarrow T_{0}^{-}}\|\partial_{y}u(t)\|_{L^2}=\infty.$
    \end{proof}
    \begin{corollary}
        The equation \eqref{fifthKPI2} is locally well-posed in $\tilde{E}^2=\{u_0\in E^2: \partial_{y}u_0 \in L^2(\mathbb{R}^2)\}.$ If $\int_{\mathbb{R}^2}y^2u_{0}^{2}<\infty,u_0\in \mathcal{J}$, then solution $u(t)$ of equation \eqref{fifthKPI2} with initial data $u_0$ blows up in finite time. 
    \end{corollary} 
    \begin{proof}
        The local-wellposedness follows from adapting the proof of Theorem \ref{mainthm}, by noting 
        the energy estimate 
        $$\partial_{t}\|\partial_{y}u(t)\|_{L^{2}_{xy}}\leq p\|u(t)\|_{L^{\infty}_{xy}}\|\partial_{x}u(t)\|_{L^{\infty}_{xy}}\|\partial_{y}u(t)\|_{L^{2}_{xy}}$$
        by applying $\partial_{y}$ of \eqref{fifthKPI2} and then multiplying it by $\partial_{y}u.$   Then we replicate the proof of Proposition \ref{welposedness} to conclude the local-wellposedness result.

        We observe that $yu_0 \in L^2(\mathbb{R}^2)$ implies $yu(t)\in L^2(\mathbb{R}^2)$ for all times of existence from Theorem 3.3 in \cite{Saut}, with $\|yu(t)\|_{L^2(\mathbb{R}^2)}\leq  C(T, \|yu_0\|_{L^2(\mathbb{R}^2)}, \|u_0\|_{E^2}),$ for all $t\in [0,T].$ Also, $u_0 \in \mathcal{J}$ implies $u(t)\in \mathcal{J}$ for all times of existence from Lemma \ref{invariance}.
        
        Then we have from Lemma \ref{blowup}, that there exists $T<\infty$ such that 
        $$\lim_{t\rightarrow T^{-}}\|u(t)\|_{\tilde{E}^2}=\infty.$$ 
        In fact, if $u\in \mathcal{J},$ then $\|u(t)\|_{E^2}\leq \sqrt{6}d_1,$ which implies that the blow-up is given by the transverse effects of the equation. 
    \end{proof}
    As a consequence of Lemma \ref{blowup}, we can state a strong instability generalization for the ground states in $G_1.$
    \begin{definition}\label{stronginstability}
        We say that the solitary wave $\phi$ is strongly unstable if for any $\delta>0$  there exists $u_0\in X_s(s\geq 5)$, close to $\phi$ in $E^2$ with $\|u_0-\phi\|_{E^2} <\delta$, such that the solution $u(t)$ of \eqref{fifthKPI} with initial data $u(0)=u_0$ blows up in finite time.
    \end{definition}
    In \cite{EsfahaniLevandovsky}, it was shown that for $p\geq 4$, then  $\varphi\in G_1$ with $y\varphi \in L^2(\mathbb{R}^2)$ is strongly unstable. We generalize this result in the following: 
    \begin{proposition}
        Suppose that $\varphi$ is a ground state of equation \eqref{fifthKPI} for $p>0$, i.e. $\varphi \in G_1$. If $y\varphi \in L^{2}(\mathbb{R}^2),$ then $\varphi$ is strongly unstable in the sense of of Definition \ref{stronginstability}. 
    \end{proposition}
    \begin{proof} 
    In the view of Theorem $4.3$ and Lemma $4.5$ of \cite{Esfahani}, it remains to show only that there exists $u_0\in E^2$ close to $\varphi$ such that $S(u_0)<d_1, I(u_0)<0, K(u_0)>0.$ Take $w_{\lambda, \xi, \mu}(x,y)=\lambda\varphi(\xi x, \mu y)$ with $\lambda=(1+\varepsilon)^{\frac{1}{2}}, \xi=(1+4\varepsilon)^{\frac{1}{4}}, \mu=(1+4\varepsilon)^{\frac{1}{2}}(1-6\varepsilon),$ for some small $\varepsilon>0.$ Then we see $S(w_{\lambda, \xi, \mu})<d_1, I(w_{\lambda, \xi, \mu})<0, K(w_{\lambda, \xi, \mu})>0.$ Blow-up is then a consequence of Lemma \ref{blowup}.
 
    \end{proof}
    \begin{remark}
        From \cite{deBouardSaut2}, we know that the ground state of equation \eqref{fifthKPI}, $\varphi \in G_1,$ has the decay rate bounded by $$|\varphi(x,y)|\lesssim \frac{1}{x^2+y^2}.$$
        This decay rate is not enough to show that the ground state satisfies $y\varphi\in L^2(\mathbb{R}^2).$ Nevertheless, we do not know a sharp decay rate for $\varphi \in G_1.$
    \end{remark}

\section{Global well-posedness for the fifth order modified KP-I equation on a cylinder}\label{RT}
In this section, we prove that global-wellposedness result for the equation 
\begin{equation}\label{eqmkp5RT}
 \begin{cases}
  \partial_t u-\partial^5_x u-\partial_x^{-1}\partial_y^2 u + \partial_x (u^3)=0,\\
u(0,x,y)=\phi(x,y)\in Z^2(\mathbb{R}\times\mathbb{T})
 \end{cases} 
 \end{equation} 
where we define the space 
$$Z^s(\mathbb{R}\times\mathbb{T})=\Big\{g:\mathbb{R}\times\mathbb{T}\rightarrow \mathbb{R}:\|g\|_{Z^s}=\norm{\Big(1+|\xi|^{2s}+\frac{|n|}{|\xi|}\Big)\hat{g}(\xi,n)}_{L^2(\mathbb{R}\times\mathbb{Z})}\Big\}.$$
We begin with Strichartz estimates that sharpens the estimates found in \cite{KenigIonescu} modifying their proofs. 
\begin{theorem}
    Assume $\phi \in H^{\infty}(\mathbb{R}\times \mathbb{T}).$ Then we have 
    \begin{itemize}
        \item[i)] 
    \begin{equation}\label{eqdispersive1}
    \|W(t)\tilde{Q}_{y}^{3j}Q_{x}^{j}\phi\|_{L^{2}_{2^{-2j}}L^{\infty}_{xy}}\lesssim 2^{-\frac{29j}{20}}\|\tilde{Q}^{3j}_{y}Q_{x}^{j}\phi\|_{L^{2}_{xy}}
    \end{equation}
    and 
    \item[ii)]
     \begin{equation}\label{eqdispersive2}
     \|W(t)Q_{y}^{3j+k}Q_{x}^{j}\phi\|_{L^{2}_{2^{-2j-k}}L^{\infty}_{xy}}\lesssim 2^{-\frac{29j}{20}-\frac{17k}{20}}\|Q^{3j+k}_{y}Q_{x}^{j}\phi\|_{L^{2}_{xy}}
     \end{equation}
     \end{itemize}
     for any integers $j\geq 0$ and $k\geq 1.$ 
\end{theorem}

\begin{proof}
    We begin with $i)$ as in the proof of Theorem 9.3.1 in \cite{KenigIonescu}. Let $a(\xi,n)=\widehat{(\tilde{Q}_{y}^{3j}Q_{x}^{j}\phi)}(\xi, n),$ so $a(0,n)=0.$ Let $\psi_0:\mathbb{R}\rightarrow [0,1]$ denote a smooth even function supported in the interval $[-2,2]$ and equal to $1$ in the interval $[-1,1]$ and $\psi_1:\mathbb{R}\rightarrow [0,1]$ a smooth even function supported in the set $\{|r|\in [\frac{1}{4},4]\}$ and equal to $1$ in the set $\{|r|\in [\frac{1}{2},2]\}.$ Then, 
    $$W(t)\tilde{Q}_{y}^{2j}Q_{y}^{j}\phi(x,y)=C\sum_{n\in \mathbb{Z}}\int_{\mathbb{R}}a(\xi,n)\psi_1\Big(\frac{\xi}{2^j}\Big)\psi_0\Big(\frac{n}{2^{3j}}\Big)e^{i(\xi x+ny+F(\xi,n)t)}d\xi$$
    where $F(\xi,n)=\xi^5+\frac{n^2}{\xi}.$
    Thus, it suffices to prove that 
    $$\|1_{[0,2^{-2j}]}(|t|)\sum_{n\in \mathbb{Z}}\int_{\mathbb{R}}a(\xi,n)\psi_1\Big(\frac{\xi}{2^j}\Big)\psi_0\Big(\frac{n}{2^{3j}}\Big)e^{i(\xi x(t)+ny(t)+F(\xi,n)t)}d\xi\|_{L^{2}_{t}}\leq C2^{-\frac{29j}{20}}\|a\|_{L^2(\mathbb{R}\times\mathbb{T})}$$
    for any measurable functions $x:[-2^{-j},2^{-j}]\rightarrow \mathbb{R}$ and $y:[-2^{-j},2^{-j}]\rightarrow \mathbb{T}.$ By duality, this is equivalent to proving that 
    $$\norm{\int_{\mathbb{R}}g(t)1_{[0,2^{-j}]}(|t|)\psi_1\Big(\frac{\xi}{2^j}\Big)\psi_0\Big(\frac{n}{2^{3j}}\Big))e^{i(\xi x(t)+ny(t)+F(\xi,n)t)}dt}_{L^2(\mathbb{R}\times\mathbb{T})}\leq C 2^{-\frac{29j}{10}}\|g\|_{L^{2}_{t}},$$
    for any $g\in L^2(\mathbb{R}).$ By expanding the $L^2$ norm in the left-hand side, it suffices to prove that 
    $$\Big|\int_{\mathbb{R}}\int_{\mathbb{R}}g(t)g(t')K_{j}(t,t')dtdt'\Big|\leq C2^{-\frac{29j}{10}}\|g\|_{L^{2}_{t}},$$
    where 
    $$K_j(t,t')=1_{[0,2^{-2j}]}(|t|)1_{[0,2^{-2j}]}(|t'|)\sum_{n\in \mathbb{Z}}\int_{\mathbb{R}}\psi_{1}^{2}(\frac{\xi}{2^j})\psi_{0}^{2}\Big(\frac{n}{2^{3j}}\Big)e^{i[\xi (x(t)-x(t'))+n(y(t)-y(t'))+F(\xi,n)(t-t')]}d\xi.$$ 
    For integers $l\geq j$ let $K_{j}^{l}(t,t')=1_{[2^{-l},2\cdot2^{-l}]}(|t-t'|)K_j(t,t').$ 
    By using Young's convolution inequality, it suffices to prove that 
    $$|K_{j}^{l}(t,t')|\leq C2^l2^{-\frac{29j}{10}}2^{\frac{2j-l}{10}}$$
    for any $l\geq 2j.$ 
    To summarize, it suffices to prove that 
    $$\Big|\sum_{n\in \mathbb{Z}}\int_{\mathbb{R}}\psi_{1}^{2}(\frac{\xi}{2^j})\psi_{0}^{2}\Big(\frac{n}{2^{3j}}\Big)e^{i[\xi x+ny+F(\xi,n)t]}d\xi\Big|\leq C2^l2^{-\frac{29j}{10}}2^{\frac{2j-l}{10}}$$
    for any $x\in \mathbb{R}, y\in [0,2\pi), |t|\in [2^{-l},2\cdot2^{-l}],l\geq 2j.$

    Let $A_j(\xi,t,y)=\sum_{n\in \mathbb{Z}}\psi_{0}^{2}\big(\frac{n}{2^{3j}}\big)e^{i(ny+\frac{n^2t}{\xi})}.$ We use the Poisson summation formula 
    $$\sum_{n\in \mathbb{Z}}F(n)=\sum_{\nu\in \mathbb{Z}}\hat{F}(2\pi\nu)$$
    for any $F \in \mathcal{S}(\mathbb{R}).$ Thus, 
    $$A_j(\xi,t,y)=\sum_{\nu\in \mathbb{Z}}\int_{\mathbb{R}}\psi_{0}^{2}\Big(\frac{\eta}{2^{3j}}\Big)e^{i(\eta(y+2\pi\nu)+\frac{\eta^2t}{\xi})}d\eta.$$
    Denote $\gamma(\eta)=\eta(y+2\pi\nu)+\frac{\eta^2t}{\xi}$ and $\zeta(\eta)=\psi_{0}^{2}\big(\frac{\eta}{2^{3j}}\big).$ 
    By integration by parts, we observe that 
    $$\int_{\mathbb{R}}\zeta(\eta)e^{i\gamma(\eta)}d\eta=\int_{\mathbb{R}}[(i\gamma')^{-1}\{(i\gamma')^{-1}[(i\gamma')^{-1}\zeta]'\}']'e^{i\gamma}d\eta$$
    $$=\int_{\mathbb{R}}[(H')^3\zeta+H (H')^2 \zeta'+3H^2H''\zeta'+3H^2H'''\zeta+3H^2H'\zeta''+H^3\zeta''']e^{i\gamma}d\eta,$$
    where $H=(i\gamma')^{-1}.$ We have that $\|\partial_{\eta}^{\alpha}\zeta\|_{L^{1}_{\eta}}\lesssim 2^{3j(1-\alpha)}\|\psi_0\|_{L^{2}_{\eta}}^{2},$ for all integers $\alpha \geq 0.$
    Since $y \in [0,2\pi), \eta \leq 2^{3j+1}, t\leq 2\cdot2^{-2j}$ and $|\xi|\geq 2^{j-1},$ then $\|\gamma'\|_{L^{\infty}_{xy}}\sim |\nu|$ for $|\nu|\geq 100.$ Also, observe that $\|\gamma''\|_{L^{\infty}_{\eta}}\sim 2^{-3j}$ and $\gamma'''=0.$ Therefore, we get that 
    $$\|H\|_{L^{\infty}_{\eta}}\sim \frac{1}{|\nu|}, \|H'\|_{L^{\infty}_{\eta}}\sim \frac{2^{-3j}}{\nu^2}, \|H''\|_{L^{\infty}_{\eta}}\sim \frac{2^{-6j}}{|\nu|^3}, \|H'''\|_{L^{\infty}_{\eta}}\sim \frac{2^{-9j}}{\nu^4}.$$
    It yields that 
    $$\Big|\int_{\mathbb{R}}\psi_{0}^{2}\Big(\frac{\eta}{2^{3j}}\Big)e^{i(\eta(y+2\pi\nu)+\frac{\eta^2t}{\xi})}d\eta\Big|\lesssim \frac{2^{-6j}}{|\nu|^3}.$$ 
    We conclude that
    $$A_j(\xi,t,y)=\sum_{|\nu|\leq 100}\int_{\mathbb{R}}\psi_{0}^{2}\Big(\frac{\eta}{2^{3j}}\Big)e^{i(\eta(y+2\pi\nu)+\frac{\eta^2t}{\xi})}d\eta+O(2^{-6j}).$$ 
    By integrating in $\xi,$ we get that 
    $$\Big|\int_{\mathbb{R}}\psi_{1}^{2}\Big(\frac{\xi}{2^j}\Big)e^{i(\xi x+\xi^5t)}O(2^{-6j})d\xi\Big|\lesssim 2^{-6j}\int_{\mathbb{R}}\psi_{1}^{2}\Big(\frac{\xi}{2^j}\Big)d\xi=2^{-5j}\|\psi_{1}\|_{L^2}^2,$$
    hence it is definitely bounded by $2^l2^{-\frac{29j}{10}}2^{\frac{2j-l}{10}}$ for $l\geq 2j.$ 
    Therefore, it remains to show that 
    $$\Big|\int_{\mathbb{R}}\psi_{1}^{2}\Big(\frac{\xi}{2^j}\Big)e^{i(\xi x+\xi^5t)}\Big(\int_{\mathbb{R}}\psi_{0}^{2}\Big(\frac{\eta}{2^{3j}}\Big)e^{i(\eta y'+\frac{\eta^2t}{\xi})}d\eta\Big)d\xi\Big|\lesssim 2^l2^{-\frac{29j}{10}}2^{\frac{2j-l}{10}},$$
    for any $l\geq 2j, x\in \mathbb{R}, y'=y+2\pi\nu \in [-1000,1000], |t|\in [2^{-l},2\cdot2^{-l}].$
    We use the formula that appeared in \cite{KenigIonescu}, namely 
    \begin{equation}\label{Fourierformula}
    \int_{\mathbb{R}}\psi_{0}^{2}\Big(\frac{\eta}{2^{3j}}\Big)e^{i(\eta y'+\frac{\eta^2t}{\xi})}d\eta=C\Big|\frac{\xi}{t}\Big|^{\frac{1}{2}}\int_{\mathbb{R}}\widehat{\psi_{0}^{2}}(\mu)e^{-i(y'+\frac{\mu}{2^{3j}})^2\frac{|\xi|}{|t|}}d\mu.
    \end{equation}
    Hence, it suffices to show that 
    $$J=\Big|\int_{\mathbb{R}}\Big|\frac{\xi}{2^j}\Big|^{\frac{1}{2}}\psi_{1}^{2}\Big(\frac{\xi}{2^j}\Big)e^{i(\xi x+\xi^5t)}\Big(\int_{\mathbb{R}}\widehat{\psi_{0}^{2}}(\mu)e^{-i(y'+\frac{\mu}{2^{3j}})^2\frac{|\xi|}{|t|}}d\mu\Big)d\xi\Big|\lesssim 2^{\frac{l-j}{2}}2^l2^{-\frac{29j}{10}}2^{\frac{2j-l}{10}},$$
    for any $x\in \mathbb{R}, y'\in [-1000,1000], |t|\sim 2^{-l}.$ From now on, we fix $x,y', t$. Define $L(\tilde{\xi})=\int_{2^{j-2}}^{\tilde{\xi}}\psi_{1}^{2}\big(\frac{\xi}{2^j}\big)\big|\frac{\xi}{2^j}\big|^{\frac{1}{2}}e^{i(\xi x+\xi^5t)}d\xi$ and $$G(\xi)=\int_{\mathbb{R}}\widehat{\psi_{0}^{2}}(\mu)e^{- i(y'+\frac{\mu}{2^{3j}})^2\frac{ \xi}{|t|}}d\mu,$$
    thus $J\lesssim \Big|\int_{0}^{\infty}L'(\xi)G(\xi)d\xi\Big|+\Big|\int_{0}^{\infty}\overline{L'(\xi)}G(\xi)d\xi\Big|,$ so it suffices to show that 
    $$\Big|\int_{0}^{\infty}L'(\xi)G(\xi)d\xi\Big|\lesssim 2^{\frac{l-j}{2}}2^l2^{-\frac{29j}{10}}2^{\frac{2j-l}{10}}.$$ 
    By integration by parts and the definition of $\psi_1$, we get that 
    \begin{equation}\label{integrationbyparts}
    \int_{0}^{\infty}L'(\xi)G(\xi)d\xi=\int_{2^{j-2}}^{2^{j+2}}L'(\xi)G(\xi)d\xi=\Big[LG\Big]_{2^{j-2}}^{2^{j+2}}-\int_{2^{j-2}}^{2^{j+2}}L(\xi)G'(\xi)d\xi.
    \end{equation}
    We have that $L(2^{j-2})=0$ and $|G(\xi)|\lesssim \|\widehat{\psi_{0}^{2}}\|_{L^1}.$ By Van der Corput's inequality (see \cite{Stein}, Chapter 8, Proposition 5), we have that for any $\tilde{\xi}\in [2^{j-2},2^{j+2}],$
    $$|L(\tilde{\xi})|=\Big|\int_{2^{j-2}}^{\tilde{\xi}}\psi_{1}^{2}\Big(\frac{\xi}{2^j}\Big)\Big|\frac{\xi}{2^j}\Big|^{\frac{1}{2}}e^{i(\xi x+\xi^5t)}d\xi\Big|\lesssim 2^{-\frac{l}{5}}(\||x|^{\frac{1}{2}}\psi_1\|_{L^{\infty}}+\|\psi_1\|_{L^2}^{2}+\|\psi_1\|_{L^2}\|\psi_1'\|_{L^2}).$$
    We also notice that 
    $$|G'(\xi)|= \Big|\int_{\mathbb{R}} \widehat{\psi_{0}^{2}}(\mu)e^{-i(y'+\frac{\mu}{2^{3j}})^2\frac{\xi}{t}}\frac{1}{2^{3j}t}\Big(2\mu y'+\frac{\mu^2}{2^{3j}}\Big)d\mu\Big|$$
    $$\approx  2^{l-3j}\Big|\int_{\mathbb{R}} \widehat{\psi_{0}^{2}}(\mu)e^{-i(y'+\frac{\mu}{2^{3j}})^2\frac{\xi}{t}}\Big(2\mu y'+\frac{\mu^2}{2^{3j}}\Big)d\mu\Big|\lesssim 2^{l-3j},$$
    where we used the decay properties of the Schwartz function $\widehat{\psi_{0}^{2}}$ and $|y'|\leq 1000.$ 
    Therefore, we obtain that $\|G'(\xi)\|_{L^{1}_{[2^{j-2},2^{j+2}]}}\leq 2^{l-2j}.$
    Plugging in \eqref{integrationbyparts} all these estimates, we get 
    $$J\lesssim 2^{-\frac{l}{5}}2^{l-2j}\leq 2^{\frac{l-j}{2}}2^l2^{-\frac{29j}{10}}2^{\frac{2j-l}{10}},$$
    since $l\geq 2j,$ which completes the proof of $i)$.  
    
    For $ii),$ we proceed the same. To conclude, it suffices to prove that 
    $$\Big|\sum_{n\in \mathbb{Z}}\int_{\mathbb{R}}\psi_{1}^{2}(\frac{\xi}{2^j})\psi_{1}^{2}\Big(\frac{n}{2^{3j+k}}\Big)e^{i[\xi x+ny+F(\xi,n)t]}d\xi\Big|\leq C2^l2^{-\frac{29j}{10}-\frac{17k}{10}}2^{\frac{2j+k-l}{10}}$$
     for any $x\in \mathbb{R}, y\in [0,2\pi), |t|\in [2^{-l},2\cdot2^{-l}],l\geq 2j+k.$
     We use the Poisson summation and the formula \eqref{Fourierformula} to reduce that we need only 
     $$\tilde{J}=\Big|\int_{\mathbb{R}}\Big|\frac{\xi}{2^j}\Big|^{\frac{1}{2}}\psi_{1}^{2}\Big(\frac{\xi}{2^j}\Big)e^{i(\xi x+\xi^5t)}\Big(\int_{\mathbb{R}}\widehat{\psi_{1}^{2}}(\mu)e^{-i(y'+\frac{\mu}{2^{3j+k}})^2\frac{|\xi|}{|t|}}d\mu\Big)d\xi\Big|\lesssim 2^{\frac{l-j}{2}}2^l2^{-\frac{29j}{10}-\frac{17k}{10}}2^{\frac{2j+k-l}{10}}.$$
     By the same reasoning as above, using Van der Corput's inequality, we obtain that 
     $$\tilde{J}\leq 2^{-\frac{l}{5}}2^{l-2j-k}\leq 2^{\frac{l-j}{2}}2^l2^{-\frac{29j}{10}-\frac{17k}{10}}2^{\frac{2j+k-l}{10}},$$
     for any $l\geq 2j+k,$ which concludes the proof of $ii).$
\end{proof}

     \begin{proposition} \label{linearestimate5RT}
 Assume $u \in C([-T,T]:H^{\infty}(\mathbb{R}\times\mathbb{T})),$ $f \in C([-T,T]:H^{\infty}(\mathbb{R}\times\mathbb{T})),$  and 
 $$[\partial_t-\partial_x^5-\partial_x^{-1}\partial_y^2]u=\partial_x f \mbox{ on } \mathbb{R}\times\mathbb{T}\times [-T,T].$$ Then, for any $\epsilon> 0$ we have
\begin{equation*} 
\begin{split} 
 \|u\|_{L^2_TL^{\infty}_{xy}}\lesssim_{\epsilon}  \|J_x^{\frac{11}{20}+\epsilon}u\|_{L^{\infty}_TL^{2}_{xy}}+\|J_x^{\frac{1}{10}-2\epsilon}J_y^{\frac{3}{20}+\epsilon}u\|_{L^{\infty}_TL^{2}_{xy}}+\|J_x^{-\frac{9}{20}+\epsilon}f\|_{L^{1}_TL^{2}_{xy}}.
 \end{split} 
 \end{equation*}
 \end{proposition} 
 \begin{proof} 
 Without loss of generality, we may assume that $u \in C^1([-T,T]:H^{\infty}(\mathbb{R}\times\mathbb{T}))$ and $f \in C([-T,T]:H^{\infty}(\mathbb{R}\times\mathbb{T})).$ It suffices to prove that if, for $\epsilon>0,$ 
\begin{equation}\label{eqdisp5}
\|\widetilde{Q}^{3j}_yQ^j_xu\|_{L^2_TL^{\infty}_{xy}}\lesssim 2^{-\epsilon j}[\|J_x^{\frac{11}{20}+\epsilon}u\|_{L^{\infty}_TL^{2}_{xy}}+\|J_x^{-\frac{9}{20}+\epsilon}f\|_{L^{1}_TL^{2}_{xy}}]
\end{equation}
 and 
 \begin{equation}\label{eqdisp5bis} 
 \|Q^{3j+k}_yQ^j_xu\|_{L^{p}_TL^{\infty}_{xy}}\lesssim  2^{-\epsilon(j+k)}[\|J_x^{\frac{1}{10}-2\epsilon}J_y^{\frac{3}{20}+\epsilon}u\|_{L^{\infty}_TL^{2}_{xy}}+\|J_x^{-\frac{9}{20}+\epsilon}f\|_{L^{1}_TL^{2}_{xy}}]
 \end{equation}
  for any integers $j\geq 0$ and $k\geq 1.$ 
 For $(\ref{eqdisp5})$, we partition the interval $[-T,T]$ into $2^{2j}$ equal subintervals of length $2T2^{-2j},$ denoted by $[a_{j,l},a_{j,l+1}),l=1,\ldots,2^{2j}.$ The term in the left-hand side of $\ref{eqdisp5}$, is dominated by 
    $$\|\widetilde{Q}^{3j}_{y}Q^{j}_{x}u\|_{L^{2}_{T}L^{\infty}_{xy}}\leq \sum_{l=1}^{2^{2j}}\|\textbf{1}_{[a_{j,l},a_{j,l+1})}(t)\widetilde{Q}^{3j}_yQ^j_xu\|_{L^p_TL^{\infty}_{xy}}$$
 By Duhamel's formula, for $t\in [a_{j,l},a_{j,l+1}],$ 
 $$u(t)=W_{(5)}(t-a_{j,l})[u(a_{j,l})]+\int_{a_{j,l}}^tW_{(5)}(t-s)[\partial_xf(s)]ds.$$ 
 It follows from the dispersive estimate $(\ref{eqdispersive1})$ that 
 \begin{equation} \label{eqint2} 
 \begin{split} 
 &\|\textbf{1}_{[a_{j,l},a_{j,l+1})}(t)\widetilde{Q}_y^{3j}Q_x^ju\|_{L^2_TL^{\infty}_{xy}}\\&\lesssim \|\textbf{1}_{[a_{j,l},a_{j,l+1})}(t)W_{(5)}(t-a_{j,l})\widetilde{Q}_y^{3j}Q_x^j u(a_{j,l})\|_{L^2_TL^{\infty}_{xy}}+\|\textbf{1}_{[a_{j,l},a_{j,l+1})}(t)\int_{a_{j,l}}^tW_{(5)}(s)\widetilde{Q}_y^{3j}Q_x^j \partial_xf(s)ds\|_{L^2_TL^{\infty}_{xy}}\\& \lesssim  2^{(-\frac{29}{20})j}\|\widetilde{Q}_y^{3j}Q_x^j u(a_{j,l})\|_{L^{2}_{xy}}+2^{(-\frac{29}{20})j}2^j\|\textbf{1}_{[a_{j,l},a_{j,l+1})}(t)\widetilde{Q}_y^{3j}Q_x^j f\|_{L^1_TL^{2}_{xy}}.
 \end{split}
 \end{equation} 
 For the first term of the right-hand side of $(\ref{eqint2})$, we have 
 \begin{equation} \label{eq17}
 \begin{split} 
 \sum_{l=1}^{2^{2j}}&2^{(-\frac{29}{20})j}\|\textbf{1}_{[a_{j,l},a_{j,l+1})}(t)\widetilde{Q}_y^{3j}Q_x^ju(a_{j,l})\|_{L^{2}_{xy}}\lesssim \sum_{l=1}^{2^{2j}}2^{(-\frac{29}{20})j}2^{-(\frac{11}{20}+\epsilon)j}\|\textbf{1}_{[a_{j,l},a_{j,l+1})}(t)\widetilde{Q}_y^{3j}Q_x^jJ_x^{\frac{11}{20}+\epsilon}u(a_{j,l})\|_{L^{2}_{xy}}\\& \lesssim 2^{2j}2^{(-2-\epsilon)j}\|\widetilde{Q}_y^{3j}Q_x^jJ_x^{\frac{11}{20}+\epsilon}u\|_{L^{\infty}_TL^{2}_{xy}} \lesssim 2^{-\epsilon j}\|J_x^{\frac{11}{20}+\epsilon}u\|_{L^{\infty}_TL^{2}_{xy}}.
 \end{split} 
 \end{equation} 
 For the second term of the right-hand side of $(\ref{eqint2})$ we have 
 \begin{equation} \label{eq18}
 \begin{split}
  \sum_{l=1}^{2^{2j}}2^{-\frac{9}{20}j}\|\textbf{1}_{[a_{j,l},a_{j,l+1})}(t)\widetilde{Q}_y^{3j}Q_x^j f\|_{L^1_TL^{2}_{xy}}&\lesssim \sum_{l=1}^{2^{2j}}2^{-\frac{9}{20}j}2^{(\frac{9}{20}-\epsilon)j}\|\textbf{1}_{[a_{j,l},a_{j,l+1})}(t)\widetilde{Q}_y^{3j}Q_x^j J_x^{-\frac{9}{20}+\epsilon}f\|_{L^1_TL^{2}_{xy}}\\& \lesssim 2^{-\epsilon j}\sum_{l=1}^{2^{2j}}\|\textbf{1}_{[a_{j,l},a_{j,l+1})}(t)\widetilde{Q}_y^{3j}Q_x^j J_x^{-\frac{9}{20}+\epsilon}f\|_{L^1_TL^{2}_{xy}}\\&\lesssim 2^{-\epsilon j}\|\widetilde{Q}_y^{3j}Q_x^j J_x^{-\frac{9}{20}+\epsilon}f\|_{L^1_TL^{2}_{xy}} \lesssim2^{-\frac{\epsilon j}{2}}\|J_x^{-\frac{9}{20}+\epsilon}f\|_{L^1_TL^{2}_{xy}}.
  \end{split}
  \end{equation}
  Therefore, $(\ref{eq17})$ and $(\ref{eq18})$ give $(\ref{eqdisp5}).$

 For $(\ref{eqdisp5bis})$, we partition the interval $[-T,T]$ into $2^{2j+k}$ equal subintervals of length $2T2^{-2j-k},$ denoted by $[b_{j,l},b_{j,l+1}),l=1,\ldots,2^{2j+k}.$ The term in the left-hand side of $(\ref{eqdisp5bis})$, is dominated by 
$$\|Q^{3j+k}_yQ^j_xu\|_{L^2_TL^{\infty}_{xy}}\lesssim 
 \sum_{l=1}^{2^{2j+k}}\|\textbf{1}_{[b_{j,l},b_{j,l+1})}(t)Q^{3j+k}_yQ^j_xu\|_{L^2_TL^{\infty}_{xy}}$$
 By Duhamel's formula, for $t\in [b_{j,l},b_{j,l+1}],$ 
 $$u(t)=W_{(5)}(t-b_{j,l})[u(b_{j,l})]+\int_{b_{j,l}}^tW_{(5)}(t-s)[\partial_xf(s)]ds.$$ 
 It follows from the dispersive estimate  $(\ref{eqdispersive2})$ that 
 \begin{equation} \label{eqint3} 
 \begin{split} 
 &\|\textbf{1}_{[b_{j,l},b_{j,l+1})}(t)Q_y^{3j+k}Q_x^ju\|_{L^2_TL^{\infty}_{xy}}\\&\lesssim   \|\textbf{1}_{[b_{j,l},b_{j,l+1})}(t)W_{(5)}(t-b_{j,l})Q_y^{3j+k}Q_x^j u(b_{j,l})\|_{L^2_TL^{\infty}_{xy}}+\|\textbf{1}_{[b_{j,l},b_{j,l+1})}(t)\int_{b_{j,l}}^tW_{(5)}(s)Q_y^{3j+k}Q_x^j \partial_xf(s)ds\|_{L^2_TL^{\infty}_{xy}}\\& \lesssim  2^{(-\frac{29}{20}j-\frac{17}{20}k)}\|Q_y^{3j+k}Q_x^j u(b_{j,l})\|_{L^{2}_{xy}}+2^{(-\frac{29}{20}j-\frac{17}{20}k)}2^j\|\textbf{1}_{[b_{j,l},b_{j,l+1})}(t)Q_y^{3j+k}Q_x^j f\|_{L^1_TL^{2}_{xy}}.
 \end{split}
 \end{equation} 
 For the first term of the right-hand side of $(\ref{eqint3})$, we have 
 \begin{equation} \label{eq19}
 \begin{split} 
 &\sum_{l=1}^{2^{2j+k}}\|\textbf{1}_{[b_{j,l},b_{j,l+1})}(t)Q_y^{3j+k}Q_x^ju(b_{j,l})\|_{L^{2}_{xy}}
 \\&\lesssim \sum_{l=1}^{2^{2j+k}}2^{(-\frac{29}{20}j-\frac{17}{20}k)}2^{(-\frac{1}{10}+2\epsilon)j}2^{-(\frac{3}{20}-\epsilon)(3j+k)}\|\textbf{1}_{[b_{j,l},b_{j,l+1})}(t)Q_y^{3j+k}Q_x^jJ_x^{\frac{1}{10}-2\epsilon}J_y^{\frac{3}{20}+\epsilon}u(b_{j,l})\|_{L^{2}_{xy}}
 \\& \lesssim 2^{2j+k}2^{(-\frac{29}{20}j-\frac{17}{20}k)}2^{(-\frac{1}{10}+2\epsilon)j}2^{-(\frac{3}{20}-\epsilon)(3j+k)}\|Q_y^{3j+k}Q_x^jJ_x^{\frac{1}{10}-2\epsilon}J_y^{\frac{3}{20}+\epsilon}u\|_{L^{\infty}_{T}L^{2}_{xy}}
 \\& \lesssim 2^{-\epsilon (j+k)}\|J_x^{\frac{1}{10}-2\epsilon}J_y^{\frac{3}{20}+\epsilon}u\|_{L^{\infty}_TL^{2}_{xy}}.
 \end{split} 
 \end{equation} 
  For the second term of the right-hand side of $(\ref{eqint3})$ we have 
 \begin{equation} \label{eq20}
 \begin{split}
  &\sum_{l=1}^{2^{2j+k}}2^{-\frac{9}{20}j-\frac{17}{20}k}\|\textbf{1}_{[b_{j,l},b_{j,l+1})}(t)Q_y^{3j+k}Q_x^j f\|_{L^1_TL^{2}_{xy}} \\&\lesssim \sum_{l=1}^{2^{2j+k}}2^{-\epsilon(j+k)}\|\textbf{1}_{[b_{j,l},b_{j,l+1})}(t)Q_y^{3j+k}Q_x^j J_x^{-\frac{9}{20}+\epsilon}J_y^{-\frac{17}{20}+\epsilon}f\|_{L^1_TL^{2}_{xy}}\\& \lesssim 2^{-\epsilon (j+k)}\sum_{l=1}^{2^{2j+k}}\|\textbf{1}_{[b_{j,l},b_{j,l+1})}(t)Q_y^{3j+k}Q_x^j J_x^{-\frac{9}{20}+\epsilon}J_y^{-\frac{17}{20}+\epsilon}f\|_{L^1_TL^{2}_{xy}}\\&\lesssim 2^{-\epsilon (j+k)}\|Q_y^{3j+k}Q_x^j J_x^{-\frac{9}{20}+\epsilon}J_y^{-\frac{17}{20}+\epsilon}f\|_{L^1_TL^{2}_{xy}}\\& \lesssim2^{-\epsilon (j+k)}\|J_x^{-\frac{9}{20}+\epsilon}J_y^{-\frac{17}{20}+\epsilon}f\|_{L^1_TL^{2}_{xy}}.
  \end{split}
  \end{equation}
  Therefore, $(\ref{eq19})$ and $(\ref{eq20})$ give $(\ref{eqdisp5bis}).$
 \end{proof} 
 
 \begin{corollary}
  Assume $u \in C^1([-T,T]:H^{\infty}(\mathbb{R}\times\mathbb{T})),$ $f \in C([-T,T]:H^{\infty}(\mathbb{R}\times\mathbb{T})),$ and 
 $$[\partial_t-\partial_x^5-\partial_x^{-1}\partial_y^2]u=\partial_x f \mbox{ on } \mathbb{R}\times\mathbb{T}\times [-T,T].$$ Then, for any $\epsilon>0$ we have
 \begin{equation*} 
 \begin{split} 
 \|u\|_{L^2_TL^{\infty}_{xy}}+\|\partial_{x}u\|_{L^2_TL^{\infty}_{xy}}\lesssim_{\epsilon}  \|J_x^{\frac{31}{20}+\epsilon}u\|_{L^{\infty}_TL^{2}_{xy}}+\|J_x^{\frac{11}{10}-2\epsilon}J_y^{\frac{3}{20}+\epsilon}u\|_{L^{\infty}_TL^{2}_{xy}}+\|J_x^{\frac{11}{20}+\epsilon}f\|_{L^{1}_TL^{2}_{xy}}.
 \end{split} 
 \end{equation*}
 \end{corollary}

 We state now a commutator estimate for the operators $J_{x}^s.$

  \begin{lemma} \label{katoponce}
\begin{itemize}
\item[(a)] (\cite{KatoPonce}, Lemma X1) Let $m\geq 0$ and $f,g \in H^{m}(\mathbb{R})$. If $s\geq1$ then 
  $$\|J^s_{\mathbb{R}}(fg)-fJ^s_{\mathbb{R}}g\|_{L^2}\leq C_s [\|J^s_{\mathbb{R}}f\|_{L^2}\|g\|_{L^{\infty}}+(\|f\|_{L^{\infty}}+\|\partial f\|_{L^{\infty}})\|J^{s-1}_{\mathbb{R}}g\|_{L^2}].$$
  \item[(b)] (\cite{KenigPonceVega}, Theorem A.12) If $s\in(0,1)$ then   $$\|J^s_{\mathbb{R}}(fg)-fJ^s_{\mathbb{R}}g\|_{L^2}\leq C_s \|J^s_{\mathbb{R}}f\|_{L^2}\|g\|_{L^{\infty}}.$$   
  \item[(c)] Let $m\geq 0$ and $f,g \in H^{m}(M)$, where $M$ is either $\mathbb{R}$ or $\mathbb{T}$. If $s>0$ then  
  $$\|J^s_M(fg)\|_{L^2}\leq C_s \|J^s_Mf\|_{L^2}\|g\|_{L^{\infty}}+ \|J^s_Mg\|_{L^2}\|f\|_{L^{\infty}}.$$
\end{itemize}
  \end{lemma}

 We are going to bound $g(T)=\|u\|_{L^2_TL^{\infty}_{xy}}^{2}+\|\partial_x u\|_{L^2_TL^{\infty}_{xy}}^{2}.$ We state the eenergy estimate: 
 \begin{lemma} \label{energyestimateRT}
 Suppose $u \in C([-T,T]:Z^s(\mathbb{R}\times\mathbb{T}))$ satisfies the initial value problem  (\ref{eqmkp5RT}), with initial data $\phi \in Z^s(\mathbb{R}\times\mathbb{T})$. Then we have $$\|u\|_{L^{\infty}_TZ^{s}}\lesssim \|\phi \|_{Z^{s}}\exp(g(T)).$$
 \end{lemma} 
 
 \begin{proof} 
  We apply to  (\ref{eqmkp5RT}), the operator $J^s_x$ and then we multiply by $J^s_x u$. By integration by parts and applying Lemma \ref{katoponce} twice we get
\begin{equation*} 
\begin{split} 
\frac{d}{dt}\|J_x^su\|^2_{L^2_{xy}}&=\int J_x^s uJ^s_x(u^2\partial_x u)=\int J^s_x u[J^s_x (u^2\partial_x u)-u^2J^s_x\partial_x u]+\int u^2J^s_x uJ^s_x \partial_x u \\& \lesssim \|J^s_x u\|^2_{L^2_{xy}}(\|u\|^2_{L^{\infty}_{xy}}+\|u\|_{L^{\infty}_{xy}}\|\partial_x u\|_{L^{\infty}_{xy}})\lesssim \|J^s_x u(t)\|^2_{L^2_{xy}}(\|u(t)\|^2_{L^{\infty}_{xy}}+\|\partial_x u(t)\|^2_{L^{\infty}_{xy}})
\end{split}
\end{equation*}
therefore, by Gr{\"o}nwall's inequality, we get that 
\begin{equation}\label{Jx}
\|J_x^s u\|_{L^{\infty}_TL^2_{xy}}\lesssim \|J_x^s\phi\|_{L^2_{xy}} \mbox{exp}(g(T)).
\end{equation}
  Again,  if we apply to (\ref{eqmkp5RT}) the operator $\partial_{x}^{-1}\partial_y$ and then we multiply by $\partial_{x}^{-1}\partial_y u$,  we obtain integrating by parts, 
  \begin{equation*}
  \frac{d}{dt}\|\partial_{x}^{-1}\partial_y u\|^2_{L^2_{xy}}=-\int_{\mathbb{R}^2} \partial_{x}^{-1}\partial_{y} u \partial_{x}^{-1}\partial_y\partial_{x}(u^3)=\int_{\mathbb{R}^2} (\partial_{x}^{-1}\partial_{y}u)^2u\partial_{x}u\leq \|\partial_{x}^{-1}\partial_{y}u\|_{L^2_{xy}}^{2}(\|u\|_{L^{\infty}_{xy}}^2+\|\partial_{x}u\|_{L^{\infty}_{xy}}^{2})
  \end{equation*} 
hence, by Gr{\"o}nwall's inequality, we get 
$$\|\partial_{x}^{-1}\partial_yu(t)\|_{L^2_{xy}}\lesssim \|\partial_{x}^{-1}\partial_y \phi\|_{L^2_{xy}}\exp(g(T))$$
which yields that $\|u\|_{L^{\infty}_TZ^{s}}\lesssim \|\phi \|_{Z^{s}}\mbox{exp}(g(T)).$ 
 \end{proof} 

 \begin{lemma} \label{aprioriRT}
    Let $u$ be a solution of \eqref{eqmkp5RT} with initial data $u_0\in H^{\infty}(\mathbb{R}\times\mathbb{T})$ with $\|u_0\|_{Z^s}$ small enough. Then, for any $s>\frac{31}{20},$ there exists $T=T(\|u_0\|_{Z^s})$ such that  
    $$\|g(T)\|_{L^{\infty}_{T}Z^s}\leq C(T, \|u_0\|_{Z^s}).$$
 \end{lemma}
\begin{proof}
     From Lemma \ref{katoponce}, by using AM-GM inequality we obtain that 
 $$\|J_{x}^{\frac{11}{20}+\epsilon}(u^3)\|_{L^{2}_{xy}}\lesssim \|J_{x}^{\frac{11}{20}+\epsilon}u\|_{L^{2}_{xy}}\|u\|_{L^{\infty}_{xy}}^{2}.$$
 Therefore, for $s>\frac{31}{20},$ by using AM-GM inequality we get that 
 $$\|u\|_{L^2_TL^{\infty}_{xy}}+\|\partial_{x}u\|_{L^2_TL^{\infty}_{xy}}\lesssim  \|u\|_{L^{\infty}_TZ^s}(1+\|u\|_{L^{2}_{T}L^{\infty}_{xy}}^2).$$
 From Lemma \ref{energyestimateRT}, we obtain 
 $g(T)\leq C(1+g(T)^2)\exp(2g(T))\|u_0\|_{Z^s}^{2},$
 and using the smallness of $\|u_0\|_{Z^s},$ by a continuity argument we obtain that 
 $$g(T)\leq C(T, \|u_0\|_{Z^s}).$$
\end{proof}

We start by stating a well-known local-wellposedness result from Iorio and Nunes (see \cite{IorioNunes}, Section 4):
  \begin{lemma}\label{IorioNunes}
 Assume $\phi \in H^{\infty}(\mathbb{R}\times \mathbb{T})$. Then there is $T = T(\|\phi\|_{H^5}) > 0$ and a solution $u \in C([-T, T ] : H^{\infty}(\mathbb{R}\times \mathbb{T}) )$ of the initial value problem 
 \begin{equation*}
\begin{cases} \partial_t u-\partial^5_x u-\partial_x^{-1}\partial_y^2 u + u^2\partial_x u=0,\\
u(0,x,y)=\phi(x,y).
\end{cases}
\end{equation*} 
  \end{lemma}
We proceed to prove the local well-posedness result. 
\begin{theorem} \label{LWP5RT}
 The initial value problem (\ref{eqmkp5RT}) is locally well-posed in $Z^{s}(\mathbb{R}\times\mathbb{T}),s>\frac{31}{20}.$ More precisely, given $u_0 \in Z^{s}(\mathbb{R}\times\mathbb{T}),s>\frac{31}{20}$, with $\|u_0\|_{Z^s}$ small enough, there exists $T=T(\|u_0\|_{Z^{s}})$ and a unique solution $u$ to the IVP such that $u \in C([0,T]:Z^{s}(\mathbb{R}\times \mathbb{T}))$, $u, \partial_x u\in L^2_TL^{\infty}_{xy}.$ Moreover, the mapping $u_0 \rightarrow u \in C([0,T]:Z^{s}(\mathbb{R}\times \mathbb{T}))$ is continuous. Also, we have conservation of the quantities 
 $$M(u(t))=\int_{\mathbb{R}\times\mathbb{T}}u(t)^2dxdy \text{  and  }E_{(s)}(u(t))=\frac{1}{2}\int_{\mathbb{R}\times\mathbb{T}}|D_{x}^su(t)|^2+[\partial_{x}^{-1}\partial_{y}u(t)]^2-\frac{1}{2}\int_{\mathbb{R}\times\mathbb{T}}u(t)^4.$$
 \end{theorem} 

\begin{proof}
\textbf{Step 1.}

Let $u_0 \in Z^{s}(\mathbb{R}\times\mathbb{T})$ and fixed $u_{0,\epsilon}\in Z^{s}(\mathbb{R}\times\mathbb{T})\cap H^{\infty}_{-1}(\mathbb{R}\times\mathbb{T})$ such that $\|u_0-u_{0,\epsilon}\|_{H^{s,s}}\rightarrow 0$ and $\|u_{0,\epsilon}\|_{Z^{s}}\leq 2\|u_0\|_{Z^{s}}.$ 

We know by Lemma \ref{IorioNunes} that $u_{0,\epsilon}$ gives a unique solution $u_{\epsilon}.$ We have by the a priori bound that $\|u_{\epsilon}\|_{L^2_TL^{\infty}_{xy}}+\|\partial_x u_{\epsilon}\|_{L^2_TL^{\infty}_{xy}}+\|\partial_y u_{\epsilon}\|_{L^2_TL^{\infty}_{xy}} \leq C_T$ and by the previous result, $\mbox{sup}_{0<t<T}\|u_{\epsilon}\|_{H^{s,s}}\leq C_T.$ 

 Henceforth, 
  \begin{equation*} 
  \begin{split} 
  \partial_t \|u_\varepsilon -u_{\varepsilon'}\|^2_{L^2}&=\int(u_\varepsilon-u_{\varepsilon'}) \partial_x(\frac{u_\varepsilon^3}{3}-\frac{u_{\varepsilon'}^3}{3})\\&=\int\partial_x(u_\varepsilon-u_{\varepsilon'})\cdot(u_\varepsilon-u_{\varepsilon'})\frac{u_{\varepsilon}^2+u_{\varepsilon}u_{\varepsilon'}+u_{\varepsilon'}^2}{3}=\\ 
  \\&=\int(u_\varepsilon-u_{\varepsilon'})^2\partial_x[\frac{u_{\varepsilon}^2+u_{\varepsilon}u_{\varepsilon'}+u_{\varepsilon'}^2}{3}]\\&\leq \| u_\varepsilon -u_{\varepsilon'}\|^2_{L^2}\big(\|u_\varepsilon\|^2_{L^{\infty}_{xy}}+\|\partial_x u_\varepsilon\|^2_{L^{\infty}_{xy}}+\|u_{\varepsilon'}\|^2_{L^{\infty}_{xy}}+\|\partial_x u_{\varepsilon'}\|^2_{L^{\infty}_{xy}}).
  \end{split}
  \end{equation*} 
   and by Gr{\"o}nwall's inequality and by Lemma \ref{aprioriRT}
   $$\|u_{\epsilon}-u_{\epsilon'}\|^2_{L^{\infty}_TL^2_{xy}}\lesssim_T\|u_{0,\epsilon}-u_{0,\epsilon'}\|^2_{L^2_{xy}},$$ hence $\mbox{sup}_{0<t<T}\|u_{\epsilon}-u_{\epsilon'}\|_{L^2_{xy}}\rightarrow 0,$ hence we can find $u \in C([0,T]:Z^{s'}(\mathbb{R}\times\mathbb{T}))\cap L^{\infty}([0,T]:Z^{s}(\mathbb{R}\times\mathbb{T}))$ with $s'<s.$ The fact that $u$ is a solution of the IVP is clear now. Uniqueness also comes from the previous Gr{\"o}nwall's inequality. 
  
\end{proof} 

\textbf{Step 2.}

We adapt the proof of Lemma $9.4.7$ of \cite{KenigIonescu}.

For $\phi \in Z^{s}(\mathbb{R}\times\mathbb{T})$ with $s>\frac{31}{20}$, let $\phi_k=P^k\phi$ where $\widehat{P^kg}(\xi,n)=\widehat{g}(\xi,n)\cdot 1_{[0,k]}(|\xi|)\cdot1_{[0,k^{s+1}]}(|n|).$ Let 
$$h_{\phi}(k)=[\sum_{n \in \mathbb{Z}}\int _{|\xi|^s+|n|^{\frac{s}{s+1}}\geq k^{s}}|\hat{\phi}(\xi,n)|^2\{1+\xi^{2s}+\frac{n^2}{\xi^2}\}d\xi]^{\frac{1}{2}}. $$ 

Clearly, $h_{\phi}$ is nonincreasing in $k$ and $\lim_{k \rightarrow \infty}h_\phi(k)=0.$
By Plancherel's theorem and the AM-GM inequality, we obtain 
\begin{equation*}
\begin{split}
\|\phi-\phi_{k}\|_{L^2_{xy}}=\|\widehat{\phi}-\widehat{\phi}_k\|_{L^2_{xy}}&=[\sum_{|n|\geq k} \int_{|\xi|\geq k}|\hat{\phi}(\xi,n)|^2d\xi]^{\frac{1}{2}}\\&\leq \Big[\sum_{n \in \mathbb{Z}}\int_{|\xi|^s+|n|^{\frac{s}{s+1}} \geq k^s}|\hat{\phi}(\xi,n)|^2\frac{[1+\xi^{2s}+\frac{n^2}{\xi^2}]}{k^{2s}}d\xi\Big]^{\frac{1}{2}} \lesssim C(\|\phi\|_{Z^s})k^{-s}h_{\phi}(k).
\end{split}
\end{equation*}
By the same observations, we have that $\|J_x^s(\phi-\phi_k)\|_{L^2_{xy}}\lesssim h_{\phi}(k)$ and by interpolation we get for $0\leq q\leq s,$ 
\begin{equation}\label{decaybound}
\|J_x^q(\phi-\phi_k)\|_{L^{\infty}_{T}L^{2}_{xy}}\lesssim h_{\phi}(k)k^{q-s}.
\end{equation}

Let $u_k\in C([0,T]:H^{\infty})$ be solutions of \eqref{eqmkp5RT} with $u_k(0)=\phi_k.$ Denote $v_{k,k'}=u_{k}-u_{k'},$ for $1\leq k\leq k'.$  By Lemma \ref{energyestimateRT} and the definition of $P_k$, we have for $p\geq s$
\begin{equation} \label{growthbound}
    \|J_{x}^{p}u_k\|_{L^{\infty}_{T}L^2}\leq k^{p-s}C(\|\phi\|_{Z^s}).
\end{equation}
From \textit{Step 1}, we get that 
\begin{equation}\label{Step1}
    \|u_k\|_{L^{\infty}_{T}Z^s}\lesssim C(T, \|\phi\|_{Z^s}).
\end{equation}

We denote $g_{k}(T)=\|u_{k}\|_{L^2_TL^{\infty}_{xy}}+\|\partial_{x}u_{k}\|_{L^2_TL^{\infty}_{xy}}$ and $g_{k,k'}(T)=\|v_{k,k'}\|_{L^2_TL^{\infty}_{xy}}+\|\partial_{x}v_{k,k'}\|_{L^2_TL^{\infty}_{xy}}.$
By Lemma \ref{aprioriRT}, we have that 
\begin{equation}\label{gTbound}
g_{k}(T)\leq C(T,\|\phi\|_{Z^s}) \mbox{ and }g_{k,k'}(T)\leq C(T,\|\phi\|_{Z^s}).
\end{equation}
We look at the difference equation 
\begin{equation} \label{difference} 
\begin{split}
&\partial_t v_{k,k'}-\partial_{x}^{5}v_{k,k'}-\partial_{x}^{-1}\partial_{y}^{2}v_{k,k'}+\partial_{x}(v_{k,k'}[u_{k}^{2}+u_{k}u_{k'}+u_{k'}^{2}])=\\
&=\partial_t v_{k,k'}-\partial_x^5v_{k,k'}-\partial_x^{-1}\partial_y^2v_{k,k'}+v_{k,k'}^2\partial_x v_{k,k'}+3u_{k}^2\partial_x v_{k,k'} \\
&+3u_{k}v_{k,k'} \partial_x u_{k}-3u_{k}v_{k,k'} \partial_x v_{k,k'} -3v_{k,k'}^{2}\partial_x u_{k}=0. 
\end{split}
\end{equation} 
If we multiply \eqref{difference} by $v_{k,k'}$ and use \eqref{decaybound}, we get 
\begin{equation}\label{L2bound}
\|v_{k,k'}\|_{L^{\infty}_{T}L^2_{xy}}\leq C(\|\phi\|_{Z^s})k^{-s}h_{\phi}(k).
\end{equation}
As a consequence of \eqref{Step1}, \eqref{L2bound} and interpolation, we notice that for $r\leq s,$
\begin{equation}\label{subJsbound}
    \|J_{x}^{r}v_{k,k'}\|_{L^{\infty}_{T}L^2_{xy}}\lesssim C(\|\phi\|_{Z^s})k^{r-s}h_{\phi}(k)^{1-\frac{r}{s}}.
\end{equation}

Again, we apply $J_x^s$ to \eqref{difference} and then we multiply by $J_x^sv_{k,k'}.$ 
By applying multiple times Lemma \ref{katoponce}, we obtain 
 \begin{equation*}
 \begin{split}
  \frac{d}{dt}\|J^s_xv_{k,k'}\|^2_{L^{2}_{xy}}&\lesssim \|J^s_x v_{k,k'}\|^2_{L^2_{xy}}\big(\|u_k\|_{L^{\infty}_{xy}}+\|\partial_{x}u_k\|_{L^{\infty}_{xy}}+\|u_{k'}\|_{L^{\infty}_{xy}}+\|\partial_{x}u_{k'}\|_{L^{\infty}_{xy}}\big)\\&+\|J^s_xv_{k,k'}\|_{L^{2}_{xy}}\|J^{s+1}_x u_{k}\|_{L^{2}_{xy}}(\|v_{k,k'}\|_{L^{\infty}_{xy}}^2+\|v_{k,k'}\|_{L^{\infty}_{xy}}\|u_{k}\|_{L^{\infty}_{xy}})\\&+\|J^s_xv_{k,k'}\|_{L^{2}_{xy}}\|J^{s}_x u_{k}\|_{L^{2}_{xy}}(\|v_{k,k'}\|_{L^{\infty}_{xy}}\|u_k\|_{L^{\infty}_{xy}}+\|\partial_xv_{k,k'}\|_{L^{\infty}_{xy}}\|u_k\|_{L^{\infty}_{xy}})\\&+\|J^s_xv_{k,k'}\|_{L^{2}_{xy}}\|J^{s}_x u_{k}\|_{L^{2}_{xy}}(\|v_{k,k'}\|_{L^{\infty}_{xy}}\|\partial_x u_k\|_{L^{\infty}_{xy}}+\|v_{k,k'}\|_{L^{\infty}_{xy}}\|\partial_x v_{k,k'}\|_{L^{\infty}_{xy}}+ \|v_{k,k'}\|_{L^{\infty}_{xy}}^2)
  \end{split}
  \end{equation*}

    We are using the following variant of Gr{\"o}nwall's inequality: 
  \begin{lemma}\label{Gronwall}
  If $\alpha(t),\beta(t)$ are two non-negative functions, and $\frac{d}{dt}u(t)\leq u(t)\beta(t)+\alpha(t)$ for all $t\in [0,T]$ then $$u(t)\leq e^{\int_0^t\beta(s)ds}\Big(u(0)+\int_0^t\alpha(s)ds\Big).$$
  \end{lemma}
Therefore, it yields 
  \begin{equation}\label{Jsdifference}
 \begin{split}
 \|J^s_x v_{k,k'}\|_{L^{\infty}_TL^2_{xy}}& \lesssim \exp(\frac{1}{2}g_{k,k'}(T)^2+\frac{1}{2}g_{k}(T)^2)\Big[\|J^s_x v_{k,k'}(0)\|_{L^{2}_{xy}}
 \\&+\|J^{s+1}_x u_{k}\|_{L^{\infty}_TL^{2}_{xy}}\|v_{k,k'}\|_{L^2_TL^{\infty}_{xy}}(\|v_{k,k'}\|_{L^2_TL^{\infty}_{xy}}+\|u_{k}\|_{L^2_TL^{\infty}_{xy}})
 \\&+\|J^{s}_x u_{k}\|_{L^{\infty}_TL^{2}_{xy}}(g_{k,k'}(T)^2+g_{k}(T)^2)\Big].
 \end{split} 
 \end{equation}

 We apply the linear estimate of Proposition \ref{linearestimate5RT} to the difference equation \eqref{difference}, then for $\epsilon<s-\frac{31}{20}$ we obtain 
\begin{equation} \label{imporveddecayrate}
\begin{split} 
 \|v_{k,k'}\|_{L^2_TL^{\infty}_{xy}}&\lesssim_{\epsilon}  \|J_x^{\frac{11}{20}+\epsilon}v_{k,k'}\|_{L^{\infty}_TL^{2}_{xy}}+\|J_x^{\frac{1}{10}-2\epsilon}J_y^{\frac{3}{20}+\epsilon}v_{k,k'}\|_{L^{\infty}_TL^{2}_{xy}}+\|v_{k,k'}(u_k^2+u_ku_{k'}+u_{k'}^2)\|_{L^{1}_TL^{2}_{xy}}\\
 &\lesssim h_{\phi}(k)k^{(-1)-}+(\|J_x^{\frac{5}{17}+\epsilon'}v_{k,k'}\|_{L^{\infty}_{T}L^{2}_{xy}})^{\frac{17}{20}-\epsilon}(\|\partial^{-1}_{x}\partial_{y}v_{k,k'}\|_{L^{\infty}_TL^2_{xy}})^{\frac{3}{20}+\epsilon}\\
 &+(g_{k}(T)^2+g_{k'}(T)^2)h_{\phi}(k)k^{-s}\\
 &\lesssim h_{\phi}(k)k^{(-1)-}+(h_{\phi}(k)^{-\frac{1}{s}}k)^{(-1-\frac{1}{16})-}C(T,\|\phi\|_{Z^s})+h_{\phi}(k)k^{-s}\lesssim k^{(-1)-},
 \end{split} 
 \end{equation}
where we used interpolation, \eqref{gTbound}, \eqref{subJsbound} and the fact that $s>\frac{31}{20}+\epsilon$. Now plugging the decay rate \eqref{imporveddecayrate} to \eqref{Jsdifference} together with \eqref{growthbound}, \eqref{gTbound} and \eqref{decaybound}, we have 
$$\|J_{x}^{s}v_{k,k'}\|_{L^{\infty}_{T}L^{2}_{xy}}\leq h_{\phi}(k)+k^{0-},$$
 which proves continuity in time of  $\|J_s^xu(t)\|_{L^2_{xy}}.$

We proceed with continuity of $\|\partial_{x}^{-1}\partial_{y}u(t)\|_{L^2_{xy}}.$ First, we notice by applying $\partial_{y}$ to the equation \eqref{eqmkp5RT} corresponding to $u_k$ and then multiply with $\partial_{y}u_k$ that 
\begin{equation}\label{partialy}
\begin{split}
&\|\partial_{y}u_k\|_{L^2_{xy}}\leq C(T, \|\phi\|_{Z^s})\|\partial_{y}\phi_k\|_{L^2_{xy}}=C(T,\|\phi\|_{Z^s})(\sum_{n\in \mathbb{Z}}\int_{\mathbb{R}}n^2\phi_{k}(\xi,n)^2d\xi)^{\frac{1}{2}}\\
&\leq C(T,\|\phi\|_{Z^s})(\sum_{n\in \mathbb{Z}}\int_{\mathbb{R}}k^2\frac{n^2}{\xi^2}\phi_{k}(\xi,n)^2d\xi)^{\frac{1}{2}}\lesssim C(T,\|\phi\|_{Z^s})k\|\partial_{x}^{-1}\partial_{y}\phi_k\|_{L^2_{xy}}\\
&\leq C(T,\|\phi\|_{Z^s})k,
\end{split}
\end{equation}
 where in the first inequality we used the definition of $\phi_k.$

We rewrite the equation \eqref{difference} as 
$$\partial_t v_{k,k'}-\partial_{x}^{5}v_{k,k'}-\partial_{x}^{-1}\partial_{y}^{2}v_{k,k'}+\partial_{x}(v_{k,k'}[3u_{k}^{2}+3u_{k}v_{k,k'}+v_{k,k'}^{2}])=0$$
and we apply to it $\partial_{x}^{-1}\partial_{y}$ and then multiply it with $\partial_{x}^{-1}\partial_{y}v_{k,k'},$ by integration by parts we obtain 
\begin{equation}
\begin{split}
\frac{d}{dt}\|\partial_{x}^{-1}\partial_{y}v_{k,k'}\|_{L^2_{xy}}&\lesssim \|\partial^{-1}_{x}\partial_{y}v_{k,k'}\|_{L^2_{xy}}[g_{k}(T)^2+g_{k'}(T)^2)+g_{k,k'}(T)^2]\\&+\|\partial_{y}u_k\|_{L^2_{xy}}\|v_{k,k'}\|_{L^{\infty}_{xy}}(\|v_{k,k'}\|_{L^{\infty}_{xy}}+\|u_k\|_{L^{\infty}_{xy}}) 
\end{split}
\end{equation}
From the Gr\"{o}nwall's inequality of Lemma \ref{Gronwall}, we have
\begin{equation}
\begin{split}
\|\partial_{x}^{-1}\partial_{y}v_{k,k'}\|_{L^{\infty}_{T}L^2_{xy}}&\lesssim \exp(\frac{1}{2}g_{k}(T)^2+\frac{1}{2}g_{k'}(T)^2)(\|\partial_{x}^{-1}\partial_{y}v_{k,k'}(0)\|_{L^2_{xy}}\\
&+\|\partial_{y}u_k\|_{L^{\infty}_{T}L^2_{xy}}\|v_{k,k'}\|_{L^2_{T}L^{\infty}_{xy}}(\|v_{k,k'}\|_{L^{2}_{T}L^{\infty}_{xy}}+\|u_k\|_{L^{2}_{T}L^{\infty}_{xy}})\\
&\lesssim C(T,\|\phi\|_{Z^s})(\|\partial_{x}^{-1}\partial_{y}v_{k,k'}(0)\|_{L^2_{xy}}+k^{0-}),
\end{split}
\end{equation}
where in the last inequality we used \eqref{gTbound} and \eqref{partialy}. Since 
$\|\partial_{x}^{-1}\partial_{y}(\phi-\phi_k)\|_{L^2_{xy}}\rightarrow 0 \mbox{ as }k\rightarrow \infty,$
we get that 
$\|\partial_{x}^{-1}\partial_{y}v_{k,k'}\|_{L^{\infty}_{T}L^2_{xy}}\rightarrow 0 \mbox{ as }k\rightarrow\infty.$ Therefore, the solution $u$ belongs to $C([0,T];Z^s).$
 
 From the anisotropic Sobolev inequality from Proposition 3 of \cite{MolinetSautTzvetkov} (also see \cite{Besov}) and \eqref{Step1}, we obtain that
 $$\|v_{k,k'}\|_{L^4_{xy}}^{4}\lesssim \|v_{k,k'}\|_{L^2_{xy}}^{3-\frac{2}{s}}\|J_{x}^{s}v_{k,k'}\|_{L^2_{xy}}^{\frac{2}{s}}\|\partial_{x}^{-1}\partial_{y}v_{k,k'}\|_{L^2_{xy}}\lesssim C(\|\phi\|_{Z^s})k^{2-3s},$$
which proves continuity in time of $\|u(t)\|_{L^4_{xy}}.$ We conclude that 
$$E_{(s)}(u(t))=E_{(s)}(\phi).$$

\textbf{Step 3.}
We conclude by proving continuity of the flow map. 
We assume that $T \in [0,\infty)$ and $\phi^l \rightarrow \phi$ in $Z^{s}(\mathbb{R}\times \mathbb{T})$ as $l\rightarrow \infty.$ We are going to prove that $u^l \rightarrow u$ in $C([0,T]: Z^{s}(\mathbb{R}\times \mathbb{T}))$ as $l \rightarrow \infty$, where $u^l$ and $u$ are solutions of the the initial value problem \eqref{eqmkp5RT} corresponding to initial data $\phi^l$ and $\phi.$ We also denote for a function $f$ the quantity $g_f(T)=\|f\|_{L^2_TL^{\infty}_{xy}}+\|\partial_{x}f\|_{L^2_TL^{\infty}_{xy}}.$

For $k\geq 1,$ let as before, $\phi^l_k=P_k\phi^l$ and $u^l_k \in C([0,T]:H^{\infty})$ the corresponding solutions. Denote by $\omega_k=u_k-u$. By the same estimates as in \eqref{imporveddecayrate} applied to $\omega_k$ and $\partial_x\omega_k$ we get 
$$\|u_k-u\|_{Z^{s}}\lesssim \mbox{exp}(\frac{1}{2}g_{\omega_k}(T)^2+\frac{1}{2}g_{u_k}(T)^2)(\|\phi_k-\phi\|_{Z^{s}}+C(T,\|\phi_k\|_{Z^{s}},\|\phi\|_{Z^{s}})k^{0-}).$$
Following the same reasoning, we have that 
$$\|u_k^l-u^l\|_{Z^{s}}\lesssim \mbox{exp}(\frac{1}{2}g_{\omega_k^l}(T)^2+\frac{1}{2}g_{u_k^l}(T)^2)(\|\phi_k^l-\phi^l\|_{Z^{s}}+C(T,\|\phi_k^l\|_{Z^{s}},\|\phi^l\|_{Z^{s}})k^{0-}).$$
Now, denote $\omega_k^l=u_k^l-u_k.$ By the same estimates from  applied to $\omega_k^l,$ then 
$$\|u_k^l-u_k\|_{Z^{s}}\lesssim \mbox{exp}(\frac{1}{2}g_{\omega_k^l}(T)^2+\frac{1}{2}g_{u_k^l}(T)^2)(\|\phi_k^l-\phi_k\|_{H^{s,s}}+C(T,\|\phi_k^l\|_{Z^{s}},\|\phi_k\|_{Z^{s}})k^{0-}).$$

The boundedness of $g_{u_k}(T), g_{u_k^l}(T), g_{\omega_k}(T)$ and $g_{\omega_k^l}(T)$ as in \eqref{gTbound} concludes that 
\begin{equation*}
\begin{split}
\|u^l-u\|_{Z^{s}}&\leq \|u_k-u\|_{Z^{s}}+\|u_k^l-u_k\|_{Z^{s}}+\|u^l_k-u^l\|_{Z^{s}}\\& \lesssim \|\phi_k-\phi\|_{Z^{s}}+\|\phi_k^l-\phi_k\|_{Z^{s}}+\|\phi_k^l-\phi^l\|_{Z^{s}}\\&+C(T,\|\phi\|_{Z^{s}},\|\phi_k\|_{Z^{s}},\|\phi^l\|_{Z^{s}},\|\phi_k^l\|_{Z^{s}})k^{0-} 
\end{split}
\end{equation*}
which, by letting $k \rightarrow \infty$, we get $\|u^l-u\|_{Z^{s}}\lesssim \|\phi^l-\phi\|_{Z^{s}}$ and proves the continuity of the flow map. 

    \begin{corollary}
    Let $u_0\in Z^2$ with $\|u_0\|_{Z^2}$ small enough. Then there exists a global in time solution $u$ of \eqref{eqmkp5RT} with initial data $u_0$ with $u\in C(\mathbb{R};Z^2).$
    \end{corollary}
    \begin{proof}
        We obtain the result by Theorem \ref{LWP5RT} and using \eqref{anisotropic} for the partially periodic setting (the proof is almost the same as in the case $\mathbb{R}\times\mathbb{R} $ with slight modifications).
    \end{proof}

\bibliographystyle{alpha} 
\bibliography{KP}
\end{document}